\documentclass[11pt]{amsart} % draft
\usepackage[T1]{fontenc}
\usepackage{enumitem}
\usepackage[dvipsnames]{xcolor}
\usepackage{mathtools}
\usepackage[pdfusetitle,bookmarksnumbered,colorlinks,
citecolor=ForestGreen,urlcolor=Bittersweet,linkcolor=WildStrawberry]{hyperref}
\usepackage{bookmark}

\usepackage[capitalize,noabbrev]{cleveref}
\crefname{enumi}{condition}{conditions}
\Crefname{enumi}{Condition}{Conditions} % Capitalized version
\crefname{exa}{Example}{Examples}
\Crefname{exa}{Example}{Examples}
\AddToHook{env/lem/begin}{\crefalias{thm}{lemma}}
\AddToHook{env/cor/begin}{\crefalias{thm}{corollary}}
\AddToHook{env/prop/begin}{\crefalias{thm}{proposition}}
\AddToHook{env/exa/begin}{\crefalias{thm}{example}}
\AddToHook{env/defn/begin}{\crefalias{thm}{definition}}

\newtheoremstyle{plain-star}{}{}{\slshape}{}{\bfseries}{}{5pt plus 1pt minus 1pt}{\thmname{#1}\thmnumber{ #2.}\mdseries\thmnote{ #3}}

\theoremstyle{plain-star}
\newtheorem{thm}{Theorem}[section]
\newtheorem*{thm*}{Theorem}
\newtheorem{prop}[thm]{Proposition}
\newtheorem{lem}[thm]{Lemma}
\newtheorem{cor}[thm]{Corollary}
\theoremstyle{definition}
\newtheorem{defn}[thm]{Definition}
\newtheorem{exa}[thm]{Example}
\theoremstyle{remark}

\newtheorem*{rem*}{Remark}

\usepackage[style=numeric,doi=false,url=false,eprint=false,isbn=false,maxnames=10]{biblatex} % ,maxalphanames=3,minalphanames=3
\AtEveryBibitem{\clearfield{series}\clearfield{volume}}
\renewbibmacro{in:}{}
\setlist{listparindent=\parindent,parsep=0pt,left=\parindent} % indentation and separation between list paragraphs
\setlist[enumerate,1]{label=(\alph*)}
\setlist[enumerate,2]{label=(\roman*)}

\newcommand{\df}{\emph}
\newcommand{\R}{\mathbb{R}}
\newcommand{\C}{\mathbb{C}}

\newcommand{\N}{\mathbb{N}}

\newcommand{\ind}{\mathbf 1}
\renewcommand{\Pr}{\mathbb{P}}

\newcommand{\E}{\mathop{}\!\mathbb{E}}

\newcommand{\Var}{\operatorname{{Var}}}
\newcommand{\Cov}{\operatorname{{Cov}}}

\newcommand{\giv}{\mathbin{\vert}}

\newcommand{\cpl}{\mathrm{c}}
\newcommand{\trp}{\mathrm{T}}

\newcommand{\inp}[2]{\langle #1, #2 \rangle}
\newcommand{\nm}[1]{\lVert #1 \rVert}
\newcommand{\opnm}[1]{\nm{#1}_{\mathrm{op}}}
\newcommand{\Fnm}[1]{\nm{#1}_{\mathrm{F}}}
\newcommand{\abs}[1]{\lvert #1 \rvert}
\newcommand{\tr}{\operatorname{tr}}

\newcommand{\grad}{\nabla}

\newcommand{\blank}{\,\cdot\,}
\newcommand{\wkconv}{\Rightarrow}
\newcommand{\eqD}{\stackrel{\mathrm{D}}{=}}

\newcommand{\Gam}{\operatorname{Gamma}}
\newcommand{\Unif}{\operatorname{Uniform}}
\newcommand{\ESD}{\operatorname{ESD}}
\newcommand{\Inv}{\operatorname{Sph}}
\newcommand{\diag}{\operatorname{diag}}
\newcommand{\PowerT}{\operatorname{PowerT}}
\newcommand{\PrincipalT}{\operatorname{PrincipalT}}
\newcommand{\ReducedT}{\operatorname{ReducedT}}
\newcommand{\MP}{\operatorname{MP}}

\title[MP law for tensor powers of exchangeable unconditional vectors]{Marchenko--Pastur law for tensor powers of exchangeable unconditional vectors}
\date{\today}
\author{Feng Cheng}
\address{Department of Mathematics\\ University of Washington\\ Seattle, WA 98195\\ USA}
\email{fecheng@uw.edu}
\author{Dan Mikulincer}
\address{Department of Mathematics\\ University of Washington\\ Seattle, WA 98195\\ USA}
\email{danmiku@uw.edu}

\calclayout

\begin{document}

\begin{abstract}
    Given an isotropic, exchangeable, and unconditional random vector $\mathbf X$, we consider the sample covariance matrix constructed from i.i.d.\ copies of several tensor models of $\mathbf X$, such as the tensor power $\mathbf{X}^{\otimes d}$. Under appropriate moment conditions on $\mathbf X$, we show that almost surely, the empirical spectral distribution converges weakly to the Marchenko--Pastur law. This extends previous results which required the coordinates of $\mathbf{X}$ to be independent. As we demonstrate, our extension applies to many new random vectors $\mathbf{X}$ of interest.
\end{abstract}
%in each dimension $n$, we obtain an exchangeable random vector in dimension $\binom{n}{d}$, where each component is the product of $d$ distinct components of $\mathbf X$. We consider the sample covariance matrix constructed from i.i.d.\ copies of these induced exchangeable vectors. Under appropriate conditions on $\mathbf X$, it is shown that the empirical spectral distribution converges weakly to the Marchenko--Pastur distribution with probability $1$. We then extend this result to a Marchenko--Pastur law for the sample covariance matrix constructed from i.i.d.\ copies of $\mathbf X^{\otimes d}$.

\maketitle
% \tableofcontents

\section{Introduction}
\subsection{Marchenko--Pastur law for tensor powers}
The \df{empirical spectral distribution} (ESD) of a real symmetric random matrix $K \in \R^{p \times p}$ is the random measure 
$$\ESD(K)\coloneqq \frac{1}{p}\sum_{j=1}^p \delta_{\lambda_j},$$ where $\lambda_1\geq\lambda_2\geq \dotsb\geq\lambda_p$ are the $p$ eigenvalues of $K$ (counting multiplicities). In the study of random matrices, much effort has been devoted to studying potential convergence of $\ESD(K)$ to a deterministic limit in various settings.

The classical theorem of \textcite{MP_1967} provides a fundamental example. Suppose we have $m = m(p)$ i.i.d.\ random vectors $\mathbf x_p^{(1)},\dotsc,\mathbf x_p^{(m)}$ in $\R^p$, with \emph{independent coordinates}. Then, under very mild assumptions, and provided that $p/m \to c$ for some constant $c > 0$, the ESD of the sample covariance matrix\footnote{This is not the usual unbiased sample covariance matrix $\frac{1}{m-1} \sum_{k=1}^m \bigl[\mathbf{x}_p^{(k)} - \E \mathbf x_p^{(k)}\bigr]\bigl[\mathbf{x}_p^{(k)} - \E \mathbf x_p^{(k)}\bigr]^\trp$ used in statistics. However, centering, a rank one change, does not affect the limiting spectral distribution because, and therefore this is the convention in random matrix theory; see \cite[page~39]{Bai_Silverstein_2010}. % When the sample vectors have complex entries, we use the conjugate transpose instead of transpose when defining the sample covariance matrix.
} $\frac{1}{m} \sum_{k=1}^m\mathbf{x}_p^{(k)}{\mathbf{x}_p^{(k)}}^\trp$ converges weakly almost surely to a deterministic distribution \begin{equation}
        \mu_{\MP(c)} = \bigl(1 - \tfrac{1}{c}\bigr)^+ \delta_0 + \frac{\sqrt{(b-x)(x-a)}}{2\pi c x} \ind_{[a,b]}(x)\,dx, \quad \text{where }a = (1 - \sqrt{c})^2 \text{ and } b = (1 + \sqrt c)^2.\label{eq:MP}
\end{equation} % The distribution in \eqref{eq:MP} is called the Marchenko--Pastur distribution. % $\C^p$ and sample covariance matrix defined using the conjugate transpose
This distribution is called the \df{standard Marchenko--Pastur law} (MP law) with aspect ratio $c$. This law characterizes how the empirical eigenvalues distribute in the proportional regime $m\asymp p$, and it has had a prominent impact across statistics and theoretical physics. The reader may look at \cite[Theorems~3.7 and 3.10]{Bai_Silverstein_2010} for two classical proofs, using the moment method and the Stieltjes transform method.

In statistical applications, independence among the coordinates (features) of each sample is often too restrictive, since many models naturally involve structured dependence. A substantial literature has therefore sought to relax independence both within and across samples, for which we refer the reader to the recent systematic treatment in \cite{yaskov2025spectra}.

In this work, we take a further step in this direction by relaxing the independence assumptions in a structured setting involving tensor products and powers. Although this setting has been considered before, existing work largely assumes independence among the variables generating each tensor sample, whereas we allow for dependence. We first describe this setting and review the relevant known results; our main results will be presented in \cref{sec:results}.

\begin{defn} \label{defn:tensor-models}
 Let $\binom{[n]}{d}$ be the collection of subsets of $[n]$ with size $d$, and let $[n]^{d}=\{1,\dotsc,n\}^d$. These index sets have sizes $\binom{n}{d}$ and $n^d$, respectively. Further, let $\nu$ be a probability measure on $\R^n$ and let $\mathbf X = (X_1,\dotsc, X_n) \sim \nu$. We define the following random tensor models.
\begin{enumerate}[label=\arabic*.]
    \item Set $p = \binom{n}{d}$ and index the coordinates of  $\R^p$ by $\binom{[n]}{d}$. We define a random vector $\mathbf x \in \R^p$ by
    \[
         \mathbf x = \biggl(\prod_{\alpha \in i} X_\alpha\biggr)_{i \in \binom{[n]}{d}}.
    \] In this case, we say $\mathbf x$ is a \df{principal tensor generated from the base distribution $\nu$}, and denote its law by $\PrincipalT(n,d,\nu)$.
   \item Set $q = n^d$, and index the coordinates of $\R^q$ by $[n]^d$. We define another random vector $\mathbf x \in \R^q$ by \[
       \mathbf x = \biggl(\prod_{j=1}^d X_{i_j}\biggr)_{i \in [n]^d},
    \]
    where $i=(i_1,\dots,i_d)\in[n]^d.$ In this case, we say $\mathbf x$ is a (full) \df{tensor power generated from the base distribution $\nu$}, and denote its law by  $\PowerT(n,d,\nu)$.
        \end{enumerate}
    With a slight abuse of notation, we will usually write $\mathbf X$ in place of $\nu$, or omit them completely when the context is clear. When we are discussing both models simultaneously, we will just call them the tensor power models.
\end{defn}
%The choices of $f$ and $g$ used to define the coordinates of $\mathbf x$ are irrelevant because they do not change the spectrum of the sample covariance matrix. If we put the $m$ samples $\mathbf x^{(1)},\dotsc,\mathbf x^{(m)}$ into a $p$-by-$m$ data matrix $\mathbb X = \begin{bmatrix}
%    \mathbf x^{(1)} & \cdots &  \mathbf x^{(m)}
%\end{bmatrix}$, then the sample covariance matrix can be expressed as $\frac{1}{m}\mathbb X \mathbb X^\trp$. Clearly permuting the rows of the data matrix $\mathbb X$ does not change the spectrum of $\frac{1}{m} \mathbb X \mathbb X^\trp$. % \fc{we can drop this, or put in the footnote perhaps}
We briefly clarify the distinction between the two models. For $\mathbf x \sim \PrincipalT(n,d,\mathbf X)$, the coordinates of $\mathbf x$ form a basis for the space of \emph{multilinear} homogeneous polynomials of degree $d$ in the coordinates $X_1,\dotsc,X_n$ of $\mathbf X$. By contrast, $\mathbf x \sim \PowerT(n,d,\mathbf X)$ means that $\mathbf x$ has the same distribution as $\mathbf X^{\otimes d}$. In the coordinates of $\mathbf x \sim \PowerT(n,d,\mathbf X)$, the monomial $X_1^{d_1}\dotsm X_n^{d_n}$ appears $d! / \prod_{\alpha=1}^n d_{\alpha}!$ times. After removing these repetitions, we obtain a basis for the homogeneous polynomials of degree $d$ in the coordinates of $\mathbf X$.\footnote{For readers familiar with tensor algebra, this is simply saying that the tensor power $\mathbf X^{\otimes d}$ belongs to the full tensor space $(\R^n)^{\otimes d}$, and, more specifically, to the symmetric subspace $\operatorname{Sym}^{d}(\R^n)$.} It will become clear that the study of the tensor power model can be reduced to the study of the principal tensor model. We will discuss this reduction in \cref{sec:ext-tensor-power}, where we will also define the \df{reduced symmetric tensor model}, which takes care of the previously mentioned repetitions. 

% samples have inherently low-dimensional structure, which maps nonlinearly into a high-dimensional space 

In our analysis of the asymptotic ESD of sample covariance matrices, we will usually be dealing with a sequence of random vectors \[\mathbf X^{(n)} = (X_1^{(n)},\dotsc,X_n^{(n)})\in \R^n,\quad \mathbf X^{(n)} \sim \nu^{(n)}.\] When it is clear from the context, we will drop the superscript ``$(n)$'' and write $\mathbf X$, $\nu$, and the components $X_1,\dotsc,X_n$.

The principal tensor model has been studied by \textcite{BVZ_2021}, \textcite{Yaskov_2023,yaskov2025remark}, and \textcite{diaconu2026empirical}, under the assumption that $\mathbf X$ is isotropic\footnote{By isotropic, we mean $\mathbf X$ has mean zero and identity covariance. In the older literature, isotropic can also mean rotationally invariant.} with independent components. Note that in these works the principal tensor model is called the \emph{(restricted) symmetric tensor model}. In particular, the following collection of results were given in \cite[Theorem~2.2 and 2.3]{Yaskov_2023} and \cite[Proposition~2.1(b)]{yaskov2025remark}. 
\begin{thm} \label{thm:indep-tensor}
    For each $n$ and corresponding $d$, set $p = \binom{n}{d}$ and consider i.i.d.\ samples \[\mathbf x_p^{(1)},\dotsc,\mathbf x_p^{(m)} \sim \PrincipalT(n,d,\nu).\] Let $K = K_p = \frac{1}{m}\sum_{k=1}^m \mathbf x_p^{(k)}{\mathbf x_p^{(k)}}^\trp$ be the sample covariance matrix, and assume $\binom{n}{d} / m \to c$. Let ``$\wkconv$'' denote weak convergence of probability measures.

    For all $n$, let $\mathbf X^{(n)} = \mathbf X = (X_1,\dotsc,X_n)$ have independent components with zero mean, unit variance, and $\sup_n \max_{1\leq \alpha\leq n} \E X_\alpha^4 < \infty$. If $d = o(\sqrt n)$, then $\ESD(K) \wkconv \mu_{\MP(c)}$ with probability $1$.

    Furthermore, for base vectors with i.i.d.\ components, we have the following complete characterization:
    \begin{enumerate}
        \item \label{enu:non-as-iff} For all $n$, let $\mathbf X^{(n)}$ have components independent and identically distributed as the same random variable $X$, which satisfy $\E X = 0$, $\E X^2 = 1$, $\E X^4 < \infty$, and also $\Pr(\abs{X} = 1) < 1$. Then $\ESD(K) \wkconv \mu_{\MP(c)}$ with probability $1$ if and only if $d = o(\sqrt n)$.
        \item \label{enu:hypercube} If $\mathbf X \sim \Unif\{-1,1\}^n$, then $\ESD(K) \wkconv \mu_{\MP(c)}$ with probability $1$ if $\min\{d,n-d\} = o(n)$. % \fc{only if is not proven, but we will take a look}
    \end{enumerate}
\end{thm}
Note that the uniform measure on the hypercube $\{-1,1\}^n$ is precisely the case when $\E X_1 = 0$ and $\abs{X_1} = 1$ a.s. Therefore, the above result settles the case when $\nu$ is a product of independent and identical distributions on $\R$.\footnote{In fact, \textcite{yaskov2025remark} showed that part~\ref{enu:hypercube} holds even for complex-valued $\mathbf X \sim \nu$ with i.i.d.\ components satisfying $\E X_1 = 0$ and $\abs{X_1} = 1$ almost surely. But we will focus only on the real case here.} We also remark that the fourth moment condition in the above theorem may be relaxed; see \cite[Theorem~2.3]{Yaskov_2023} for details. 

Furthermore, \textcite[Section~3]{Yaskov_2023} considered the sample covariance matrix $K$ built from $m$ independent samples distributed according to $\PowerT(n,d,\nu)$, again assuming that $\nu$ is an isotropic product measure. Roughly speaking, he showed that the limiting ESD for the \emph{nonzero} eigenvalues of $K$ can still be expressed by the MP law when $d = o(\sqrt n)$.

% For base vector $\mathbf X$ with unbounded norm, the tail of the tensors become much heavier as $d$ increases. \fc{consider dropping}
We observe that when $\mathbf X$ has independent coordinates, the coordinates of the principal tensor and the tensor power induced by $\mathbf X$ are independent only when $d = 1$. As $d$ increases, the coordinates will become more correlated, especially if we allow $d$ to scale with $n$. The random tensor models are thus a structured but highly correlated model.

\Cref{thm:indep-tensor} shows that the dependence created by taking tensor powers does not, by itself, prevent convergence to the MP law. The key point, however, is that the base vector $\mathbf X$ has independent coordinates, so all dependencies arise from the algebraic relations among the coordinates of the tensor power. It is therefore natural to ask to what extent the independence assumption on $\mathbf X$ can be relaxed while preserving the same limiting behavior. This question is also relevant in applications, where models allowing dependence among the coordinates of the underlying data are often more realistic than the product models. 

In this paper, we address this question by studying the MP law for tensor models generated from base vectors with dependent components. We focus on exchangeable and unconditional distributions, which, on the one hand, allow for various interesting forms of dependence and, on the other hand, preserve enough symmetry to reduce the analysis to a tractable collection of mixed moments. In terms of these moments, our main result gives a general criterion for the sample covariance spectrum of i.i.d.\ principal tensors to converge to the MP law. We derive more readily verifiable sufficient conditions in several settings and apply them to natural classes of dependent distributions. Moreover, our analysis extends from principal tensors to the full tensor powers $\mathbf X^{\otimes d}$, which appears prominently in applications. In many cases, convergence holds for $d=o(\sqrt{n})$, matching the optimal range known for general i.i.d.\ base vectors.

%The aim of this paper is to further relax the model and remove the assumption that the base vector $\mathbf X$ has independent coordinates. We instead assume that the features of the base vector are invariant under permutations and sign changes to the coordinates. 
% When taking the tensor power of a \emph{single} vector it makes sense to consider vectors without independence between coordinates. 
%In addition, our exchangeable setup still allows us to consider limiting ESD for the full tensor power, which appears prominently in applications. We believe our analysis also contribute to the study of tensor powers, random chaos, and high-dimensional exchangeable and unconditional distributions.

\subsection{Acknowledgment} F.C.\ was supported in part by NSF grant DMS-1954059 and the McFarlan Fellowship from the Department of Mathematics at the University of Washington. D.M.\ was partially supported by the Brian and Tiffinie Pang Faculty Fellowship.

\section{Main results} \label{sec:results}
\subsection{A general condition for exchangeable and unconditional base vectors} \label{sec:general-cond-exchange-uncond-base}
As discussed above, we focus on tensor models generated from base vectors that are both exchangeable and unconditional. We begin by defining these two symmetry properties and then state our general condition for convergence to the Marchenko--Pastur law. The proof is given in \cref{sec:concentration-proof-of-main}. %In \cref{sec:concentration-proof-of-main}, we will adapt the proof of \cref{thm:indep-tensor} for i.i.d.\ base vectors in \cite{Yaskov_2023} to the case where the base vector $\mathbf X$ is exchangeable and unconditional.

We say that $\mathbf X$ is \df{exchangeable} if for any permutation $\sigma$ on $[n]$, it holds that \[
    (X_1,\dotsc,X_n) \eqD (X_{\sigma(1)},\dotsc,X_{\sigma(n)}).
\] Moreover, we say that $\mathbf X$ is \df{unconditional} if for independent Rademacher random signs $\{\varepsilon_i\}_{i=1}^n$ (independent of $\mathbf X$), we have \[
    (X_1,\dotsc,X_n) \eqD (\varepsilon_1X_1,\dotsc,\varepsilon_nX_n).
\] Therefore, assuming $\mathbf X$ is exchangeable and unconditional is simply saying that the distribution of $\mathbf X$ is invariant under any permutations and sign changes to its coordinates.

% We usually implicitly assume the aspect ratio assumption $\binom{n}{d} / m \to c$ when stating that the almost sure weak convergence $\ESD(K_{\PrincipalT}) \wkconv \mu_{\MP}$.

\begin{thm} \label{thm:main}
    Define $p = \binom{n}{d}$. Let $m=m(p) \in \N$ be the sample size, and $\mathbf X$ be an exchangeable and unconditional base vector in $\R^n$. Suppose that $p / m \to c$ for some fixed constant $c > 0$, and that for $1 \leq k\leq m$, the samples $\mathbf x^{(k)} = \mathbf x^{(k)}_p$ are i.i.d.\ according to $\PrincipalT(n,d,\mathbf X)$.

    Let $L = L_n$ and $d = d_n$ be two sequences that satisfy $Ld^2 / n \to 0$. Suppose \begin{enumerate}[leftmargin=*,label=(\Alph*)]
        \item $\E(X_1^2\dotsm X_d^2) \to 1$; \label{cond:expec}
        \item $\E(X_1^2\dotsm X_{2d}^2) - \bigl(\E X_1^2\dotsm X_d^2\bigr)^2 \to 0$; \label{cond:var}
        \item \label{cond:2nd-4th-compare} for all $n$ large enough and for all $1\leq r\leq d$, \[\E(X_1^4\dotsm X_r^4 X_{r+1}^2\dotsm X_{2d - r}^2) \leq L^r.\]
    \end{enumerate} Then for the sample covariance matrix $K_{\PrincipalT}= \frac{1}{m}\sum_k \mathbf x^{(k)}{\mathbf x^{(k)}}^\trp$, with probability $1$ we have \[\mathrm{ESD}(K_{\PrincipalT}) \wkconv \mu_{\MP(c)},\] where $\mu_{\mathrm{MP(c)}}$ is defined in \eqref{eq:MP}.
\end{thm}

In light of \cref{thm:main}, we view the assumptions of having $\mathbf X$ exchangeable and unconditional as a convenient condition that allows for a tractable analysis of $\PrincipalT(n,d,\mathbf X)$. For each multi-index of exponents $\alpha=(\alpha_1,\dotsc,\alpha_j)$, exchangeability allows us to look only at the mixed moment in the first $j$ coordinates \[\E(X_1^{\alpha_1} X_2^{\alpha_2} \dotsm X_j^{\alpha_j}).\] 
In addition, because of the unconditional assumption, we no longer need to consider the case when odd numbers appear in the exponents of the mixed moment above. This was a major difficulty in the original proof of \cref{thm:indep-tensor} in \cite{Yaskov_2023}, which was resolved thanks to the independence between the coordinates that \citeauthor{Yaskov_2023} assumed.

\begin{rem*}
In the exchangeable and unconditional setting, we believe the conditions in \cref{thm:main} should be optimal. Take $d = 1$, so $p = n$ and $\mathbf x = \mathbf X$ is simply an exchangeable and unconditional vector. We assume $\E X_1^2 \to 1$ and $\E X_1^4 = o(p)$, which correspond to conditions~\ref{cond:expec}\ref{cond:2nd-4th-compare}. This setting has actually been studied before: \cite[Proposition~3.8]{yaskov2025spectra} and \cite[Proposition~2.6]{Adamczak_2013} state that (assuming the stronger $\sup_p\E X_1^4 < \infty$) a largely necessary condition for the sample covariance spectrum to converge to MP is for \begin{equation}\frac{\nm{\mathbf x}_2^2 - \E \nm{\mathbf x}_2^2}{p}\to 0 \quad\text{in probability}.\footnote{We always write $\nm{\mathbf x}_k$ for the $\ell^k$ norm of the vector $\mathbf x$, and $\nm{X}_{L^q(\Pr)} = (\E \abs{X}^q)^{1/q}$ for the random variable $X$.} \label{eq:exchangeable-uncond-sufficient}\end{equation} Most successful attempts at relaxing the within-sample independence for explicit distributions is to prove the weaker $L^2$ convergence. By Chebyshev's inequality, we wish \begin{align*}
    \Var(\nm{\mathbf x}_2^2) & = \E\bigl(X_1^2+\dotsm +X_p^2\bigr)^2 - \bigl(p \E X_1^2\bigr)^2\\
    & = p \E X_1^4 + p(p-1)\E (X_1^2X_2^2) - p^2 \bigl(\E X_1^2\bigr)^2\\
    & = p \bigl[\E X_1^4 - \bigl(\E X_1^2\bigr)^2\bigr] + (p^2 - p)\bigl[\E(X_1^2X_2^2) -  \bigl(\E X_1^2\bigr)^2\bigr]
\end{align*}
to be $o(p^2)$. This requires $\E(X_1^2X_2^2) - \bigl(\E X_1^2\bigr)^2\to 0$, which is precisely \cref{cond:var}. % Note that $\E X_1^4 < \infty$ is precisely \cref{cond:2nd-4th-compare}, and these justify our conditions are not likely to be further relaxed.
% \dm{Say something like this seems like a hard condition to check in practice, but the point is that exchangeability makes it tractable. Below we separate into the two regimes and explain how to understand it in each one.}
\end{rem*}
The three conditions of \cref{thm:main} should really be seen as the convergence of triangular arrays of exchangeable random variables. They are hard to check in practice except in a few cases, such as when the base vector $\mathbf X$ obeys certain symmetries with respect to the $\ell^k$ norm, or when $\mathbf X$ follows a mixture of product measures, which we will see soon. But again thanks to the exchangeability assumption, we can make these conditions manageable for more general distributions.

Below we split into the case when $d$ is fixed (along with the $L$ in \cref{thm:main}) and the case when $d$ is allowed to grow with $n$. In each case, we find sufficient conditions, which are easy to verify, and which imply the conditions in \cref{thm:main}. Throughout the remainder of this section, unless stated otherwise, we let $p=\binom{n}{d}$, assume that $p/m\to c$, and let $K_{\PrincipalT}$ be the sample covariance matrix formed from $m$ i.i.d.\ samples from $\PrincipalT(n,d,\mathbf X)$.

% Since de~Finetti's theorem only applies to a single infinite sequence of exchangeable random variables, we certainly cannot reduce a finite exchangeable random vector into a mixture of i.i.d.\ distributions. However, approximation by mixture distributions is still possible for each finite exchangeable vector, as discussed in \cite{Diaconis_Freedman_1980}. It turns out we can relax our conditions slightly to obtain necessary conditions on the norm $\mathbf X^{(n)}$, which has been subject to intense study in high-dimensional probability.

\subsection{Verifying Theorem \ref{thm:main} for fixed \texorpdfstring{$d$}{d} and \texorpdfstring{$L$}{L}}
In \cref{sec:finite-d} we will focus on the case when $d$ is a fixed number, independent of $n$. Our condition earlier in \cref{thm:main}\ref{cond:2nd-4th-compare} appears to be involved, but thanks to the exchangeability assumption, in the case where $L$ and $d$ are both finite, this condition is completely equivalent to saying $\sup_n \E(X_1^4\dotsm X_d^4) < \infty$; see \cref{lem:relax-4th-moment}. % \fc{and from there we derive 2.2(c), move before}
From there we derive a set of easier conditions sufficient for applying \cref{thm:main}. % \fc{I need to let $L$ be finite as well, this is an important change from previous drafts}

%With the only exception \cref{thm:tensor-power}, we will always implicitly assume the setup in the first paragraph of \cref{thm:main} in our results for the principal tensor sample covariance matrix $K_{\PrincipalT}$. In particular, this includes the aspect ratio assumption $\binom{n}{d} / m \to c$.

\begin{thm} \label{thm:finite-case}
    Let $d$ be fixed, and let $\mathbf X = (X_1,\dotsc,X_n)$ be the exchangeable unconditional base vector. If 
    \begin{enumerate}
        \item $\E X_1^2 \to 1$, 
        \item \label{cond:asymp-uncorr} $\E(X_1^2X_2^2) - \E(X_1^2)\E(X_2^2) \to 0$, and
        \item \label{cond:fixed-d-integrability} there exists some $\epsilon > 0$ such that $\sup_n \E(\abs{X_1}^{4d + \epsilon}) < \infty$,%\footnote{I have tried, but don't think we can improve to uniform integrability}
    \end{enumerate} then the three conditions in \cref{thm:main} are satisfied, and we have with probability $1$ that \[\ESD(K_{\PrincipalT}) \wkconv \mu_{\MP(c)}.\]
\end{thm}
When $d$ is finite, we also point out a considerable generalization of \cref{thm:main} is possible. In this generalization we weight each independent base vector $\mathbf X$ by an independent random variable, and obtain an \emph{anisotropic Marchenko--Pastur law}.
\begin{thm} \label{thm:weighted-sample}
    Let $d$ be fixed. Suppose the base vector $\mathbf X = (X_1,\dotsc, X_n)$ \begin{enumerate}[label=\arabic*.]
        \item has independent isotropic components with $\sup_n \max_{1\leq \alpha \leq n} \E X_\alpha^4 < \infty$ (as in \cref{thm:indep-tensor}); or 
        \item is an exchangeable and unconditional vector that satisfies the three conditions in \cref{thm:main}.
    \end{enumerate}
    Assume in addition that we have a real-valued random variable $R$, independent of $\mathbf X$.
    
    Consider $\mathbf x^{(1)},\dotsc,\mathbf x^{(m)} \in \R^p$ sampled independently from $\PrincipalT(n,d, R\cdot \mathbf X)$, and suppose $p / m \to c$ for some fixed $c > 0$. Then with probability $1$, the sample covariance matrix $K = \frac{1}{m} \sum_{k=1}^m \mathbf x^{(k)} {\mathbf x^{(k)}}^\trp$ converges weakly to a weighted MP law $\mu$, which is uniquely determined by its Stieltjes transform \[
      s(z) = \frac{1}{\int_0^\infty \frac{x}{1 + c x s(z)}\,d\tau(x) - z},
    \] where $R^{2d} \sim \tau$.
\end{thm}

\subsection{Verifying Theorem \ref{thm:main} for diverging \texorpdfstring{$d$}{d}}
One may have noticed that it is generally not hopeful to extend \cref{thm:finite-case} to the case when $d \to \infty$, since the term $\E(\abs{X_1}^{4d + \epsilon})$ will explode in \cref{thm:finite-case}\ref{cond:fixed-d-integrability}. Therefore we need to relax the three conditions in \cref{thm:main} differently when $d$ is allowed to grow, which will be discussed in \cref{sec:growing-d}. The following result relaxes the condition $\E(X_1^2\dotsm X_d^2)\to 1$ to the convergence of the $2$-norm of the entire base vector $\mathbf X$.

\begin{thm} \label{thm:the-norm-cond}
    Assume $d = o(\sqrt n)$. Suppose the exchangeable unconditional base vector $\mathbf X$ satisfies \begin{enumerate}[label=(\roman*)]
        \item \label{cond:isotropic-growing-d} $\E X_1^2 = 1$,
        \item \label{cond:2-norm-conv} $\frac{1}{n^{4d}}\E\nm{\mathbf X}_2^{8d} \to 1$, and 
        \item \label{cond:8th-block-mmt} there exists a sequence $L = L_n$ satisfying $Ld^2 / n \to 0$ such that $\E(X_1^8\dotsm X_r^8) \leq L^{2r}$ for all $1 \leq r\leq d$.
        % \item $\nm{\mathbf X}_2$ is subexponential with proxy variance uniformly bounded for all $n$,
        % \item $\nm{\mathbf X}_4^2$ is subexponential with proxy variance uniformly bounded for all $n$.
    \end{enumerate} then the three conditions in \cref{thm:main} are satisfied, and we have with probability $1$ that \[\ESD(K_{\PrincipalT}) \wkconv \mu_{\MP(c)}.\]

    Alternatively, \cref{cond:8th-block-mmt} may be replaced by 
    \begin{enumerate}[resume,label=(iii')]
        \item \label{cond:4-norm-bdd} there exists a sequence $L = L_n$ satisfying $Ld^2 / n \to 0$ such that $\frac{1}{n^{2r}} \E\nm{\mathbf X}_4^{8r} \leq L^{2r}$ for all $1 \leq r \leq d$.\end{enumerate}
\end{thm}
Remark that, as opposed to \cref{thm:main}\ref{cond:expec}, \cref{thm:the-norm-cond}\ref{cond:isotropic-growing-d} requires exact isotropcity. Other than that, the new set of conditions in \cref{thm:the-norm-cond} significantly weaken \cref{thm:finite-case}\ref{cond:fixed-d-integrability}. Meanwhile, we strengthen \cref{thm:finite-case}\ref{cond:asymp-uncorr}, which required the second moments to be asymptotically uncorrelated. Below we demonstrate one natural and useful way to apply \cref{thm:the-norm-cond}. The idea is that proper concentration of the $2$- and $4$-norms can imply the necessary conditions.

\begin{prop} \label{prop:subexponential-Gaussian-norm}
    Assume $\mathbf X$ is isotropic, exchangeable, unconditional, with the component satisfying $\sup_{n} \E X_1^4 < \infty$. Assume in addition that $\nm{\mathbf X}_2$ and $\nm{\mathbf X}_4$ are both subexponential with uniformly bounded proxy variance $2\sigma^2$, i.e., for $k = 2,4$ and $ \lambda^2 \leq 1/\sigma^2$, \[
    \E\exp\bigl(\lambda (\nm{\mathbf X}_k - \E \nm{ \mathbf X}_k)\bigr) \leq \exp(\lambda^2 \sigma^2).
\] Then given $d = o(n^{1/3})$, we have $\ESD(K_{\PrincipalT}) \wkconv \mu_{\MP(c)}$ with probability $1$. % both the uniform bounded proxy variance and the component 4th moment bound can be relaxed in the subexponential case

% \fc{I have tried improving this, but there seems to be no good tools to do so. An example that shows the sharpness would be good}
    % \fc{There is a chance we can improve this to $o(n^{1/3})$ by (iv) above. Apparently in the worst case for Inv2 distributions $\sqrt n + $ Laplace, which satisfies Poincare?, the inequality is true for $d = o(\sqrt n)$.}
    
    Assume furthermore that $\nm{\mathbf X}_4$ is sub-Gaussian with uniformly bounded proxy variance $2\sigma^2$, i.e., for all $\lambda \in \R$, \[
        \E\exp\bigl(\lambda (\nm{\mathbf X}_4 - \E \nm{ \mathbf X}_4)\bigr) \leq \exp(\lambda^2 \sigma^2).
    \] Then given $d = o(n^{1/2})$, we have $\ESD(K_{\PrincipalT}) \wkconv \mu_{\MP(c)}$ with probability $1$.
\end{prop}

\subsection{Extension to the tensor power model} \label{sec:ext-tensor-power}
\cref{thm:finite-case} and \cref{thm:the-norm-cond} not only provide us with sufficient conditions to handle the the principal tensor models, but also allow our result to extend to the tensor power models. We base this extension on the ideas that appeared in \cite[Section~3]{Yaskov_2023}, and prove it in \cref{sec:tensor-power}.

For the sake of intuition, let $\mathbf x^{(1)},\dotsc,\mathbf x^{(m)}$ be i.i.d.\ according to $\PowerT(n,d,\mathbf X)$. We need to describe the spectrum of the sample covariance matrix $K_{\PowerT} = \frac{1}{m}\sum_{k=1}^m \mathbf x^{(k)}{\mathbf {x}^{(k)}}^\trp$: the eigenvalues have the form \[\lambda_1 \geq \dotsm \geq \lambda_{\binom{n+d-1}{d}} \geq \lambda_{\binom{n+d-1}{d} + 1} = \dotsb  = \lambda_{n^d} = 0.\] It is not hard to understand why most eigenvalues of $K_{\PowerT}$ are zero, since the same monomial can repeat multiples times in the coordinates of the tensor $\mathbf x = \mathbf X^{\otimes d}$ (as discussed after \cref{defn:tensor-models}). The following result says that the ESD of the largest $\binom{n+d-1}{d}$ eigenvalues, which correspond to the nontrivial values, can be described asymptotically by the MP law.

\begin{thm} \label{thm:tensor-power}
    Assume the sequence of base vectors $\mathbf X^{(n)} \in \R^n$ satisfies the conditions of \cref{thm:finite-case} for fixed $d$, or \cref{thm:the-norm-cond} for $d = o(\sqrt n)$. Suppose that $\binom{n}{d} / m \to c$ for some constant $c > 0$, and $\mathbf x^{(1)},\dotsc,\mathbf x^{(m)}$ be i.i.d.\ according to $\PowerT(n,d,\mathbf X)$. Then it holds with probability $1$ that \[\frac{1}{\binom{n+d-1}{d}}\sum_{j=1}^{\binom{n+d-1}{d}} \delta_{\lambda_j / d!} \wkconv \mu_{\MP(c)}.\]
\end{thm} % \dm{Maybe state theorem first and them supply intuition and explain the proof}

We remark that due to the $d = o(\sqrt n)$ assumption, to say  $\binom{n}{d} / m \to c$ is equivalent to saying $\binom{n+d-1}{d} / m \to c$ above; see \cref{lem:sqrt-n-lim}. % \fc{consider possibly adding a sentence saying that we state all results for principal, and leave extension to tensor power to the reader when theorem 2.2 2.4 are satisfied. But I think this is probably going to be more confusing}

We outline some necessary ideas from \cite[Section~3]{Yaskov_2023}, which will appear in the proof. To study the tensor power model with feature size $n^d$, we need to look at the \emph{reduced symmetric tensor model} $\ReducedT(n,d,\mathbf X)$, where $p = \binom{n+d-1}{d}$, and each sample $\mathbf x_p$ consists of entries \[x_{d_1,\dotsc,d_n} = \prod_{\alpha = 1}^n \frac{X_\alpha^{d_\alpha}}{\sqrt{d_\alpha !}}, \quad \text{where }\sum_{\alpha =1}^n d_\alpha = d.\]

This definition might seem obscure, and hence we provide some intuition. Notice that the law $\ReducedT(n,d,\mathbf X)$ can be obtained from $\PowerT(n,d,\mathbf X)$ as follows: for each $t = (t_1,\dotsc,t_d) \in [n]^d$, we have a corresponding entry $\prod_{i=1}^d X_{t_i}$ in the tensor power model. One can then see the relation \[\frac{\prod_{i=1}^d X_{t_i}}{\prod_{\alpha=1}^n \sqrt{d_\alpha !}} = \prod_{\alpha=1}^n  \frac{X_\alpha^{d_\alpha}}{\sqrt{d_\alpha !}}, \quad \text{where }d_\alpha = \sum_{i=1}^d \ind\{\alpha = t_i\}.\] Note that $d! / \prod_{\alpha }d_\alpha!$ counts how many times the entry $\prod_{\alpha = 1}^n X_{\alpha}^{d_{\alpha}}$ exactly appears in the (full) tensor power model. Therefore, up to a factor of $d!$, the reduced symmetric tensor model aims to remove the repetition of entries in the tensor power model, while retaining the sample covariance structure. 

The sample covariance matrix $K_{\PowerT}$ for the tensor power model has at most $p = \binom{n+d-1}{d}$ nonzero eigenvalues. From the intuition in the above paragraph, it should not be hard to see that these $p = \binom{n+d-1}{d}$ eigenvalues of $K_{\PowerT}$ are precisely $d!$ times the $p$ eigenvalues of the sample covariance $K_{\ReducedT}$ for the reduced symmetric tensor model; see \cite[Proposition~3.2]{Yaskov_2023} for the proof. Therefore, the study of the spectrum $K_{\PowerT}$ is equivalent to the study of $\ESD(K_{\ReducedT})$. To complete the proof, in \cref{sec:tensor-power}, we will sketch prove that $\ESD(K_{\ReducedT}) \wkconv \mu_{\MP(c)}$ with probability $1$. This will follow from our previous analysis on $\ESD(K_{\PrincipalT})$.

\subsection{Examples and applications}
% \fc{some 2.1 some 2.2 2.4, in which case extend by 2.6}
% We provide some justification that shows the three conditions cannot be further relaxed. First, convergence of the block second moment, \cref{cond:expec}, is necessary because we need the population covariance matrix to have converging spectrum. Second, \cref{cond:2nd-4th-compare} is fundamental, due to the following observation.

% \dm{First have the condition for $d$ finite. Theorem \ref{thm:finite-case}, then we can have Corrollary 1.4, and then Theorem \ref{thm:the-norm-cond}.}

% For the special case where $X_1,\dotsc,X_n$ are furthermore independent, \cref{thm:main}\ref{cond:expec} becomes $\bigl(\E X_1^2\bigr)^d \to c$. Therefore, if $d$ is infinite, then the only possibility is for $c = 1$ and $\E X_1^2 \to 1$. Therefore we will assume $c = 1$ henceforth without loss of generality, since we can always divide $X_1$ by $c^{1/d}$ when $d$ is finite. \fc{edit}

We now look at specific base vectors $\mathbf X \sim \nu$ such that the three conditions of \cref{thm:main} are satisfied. The verification of these examples are covered in \cref{sec:examples}. Note that for brevity we state all results for $K_{\PrincipalT}$, but thanks to \cref{thm:tensor-power}, for base vectors that satisfy conditions of \cref{thm:finite-case} and \cref{thm:the-norm-cond}, one can in addition state a convergence result for the spectrum of the nontrivial eigenvalues of $K_{\PowerT}$.

\subsubsection{Approximation to mixture distributions}
We begin with an easy example in which $\mathbf X$ follows a mixture of product distributions. By de~Finetti's theorem, this covers every exchangeable vector which extends to an infinite exchangeable sequence. In this case, we can factor the moments $\E(X_1^4\dotsm X_r^4 X_{r+1}^2\dotsm X_{2d-r}^2)$ appearing in \cref{thm:main} and verify the conditions directly. We prove this as a motivating example at the beginning of \cref{sec:finite-d}.

\begin{prop} \label{prop:mixture}
    Assume $d = o(\sqrt n)$ and $L$ to be some positive constant. Let $\mathbf X^{(n)}$ follow any mixture $\otimes^n \theta\,dw_n$ of the product of i.i.d.\ symmetric distributions $\theta$, with $\int x^2\,d\theta = 1$ and $\int x^4\,d\theta \leq L$ for $w_n$-a.e.\ $\theta$. Then we have with probability $1$ that $\ESD(K_{\PrincipalT}) \wkconv \mu_{\MP(c)}$.
\end{prop}
\subsubsection{Uniform signed permutations}
The next example is a natural choice for a discrete exchangeable random vector which cannot be extended to an infinite exchangeable sequence in general. In this result, we consider a uniform random permutation of a fixed vector. Since it is very straightforward to verify the conditions of \cref{thm:the-norm-cond}, we give all of the necessary details here.
\begin{exa}
    Assume $d = o(\sqrt n)$. For each $n$, let $\mathbf c = (c_1,\dotsc,c_n)$ be a fixed vector in $\R^n$, where we require $\nm{\mathbf c}_2^2 = n$ and $\sup_n \frac{1}{n}\nm{\mathbf c}_4^4 < \infty$. Now let $\mathbf X$ follow the uniform distribution on the \df{signed permutations} of the vector $\mathbf c = (c_1,\dotsc,c_n)$: \[\bigl\{(a_1c_{\sigma(1)},\dotsc,a_nc_{\sigma(n)}) : \text{each } a_j = \pm 1 \text{ and }\sigma \in S_n\bigr\}.\]

    Then $\E X_1^2 = \frac{1}{n}\E\bigl(\sum_{j=1}^n X_j^2\bigr) = \frac{1}{n}\nm{\mathbf c}_2^2 = 1$ and \[
    \frac{1}{n^{4d}} \E \nm{\mathbf X}_2^{8d} = \frac{1}{n^{4d}} \nm{\mathbf c}_2^{8d} = 1.
\] If we define $L = \sup_n \frac{1}{n}\nm{\mathbf c}_4^4$, then for all $n$ and $1 \leq r \leq d$, we have \[
    \frac{1}{n^{2r}} \E\nm{\mathbf X}_4^{8r} = \frac{1}{n^{2r}}\nm{\mathbf{c}}_4^{8r} \leq L^{2r}.
\] We have thus verified \cref{thm:the-norm-cond} conditions \ref{cond:isotropic-growing-d}\ref{cond:2-norm-conv}\ref{cond:4-norm-bdd}. Further note that by taking $\mathbf c = (1,\dotsc,1)$, we recover \cref{thm:indep-tensor}\ref{enu:hypercube} for $d = o(\sqrt n)$.
\end{exa}

 % This example shows that for fixed $d$, when the base vector is uniformly distributed on signed permutations of vectors other than $(1,\dotsc,1)$, the MP law for the principal tensor model may still be in force. \fc{maybe delete the last sentence, awkward}

\subsubsection{Continuous distributions}
It turns out that many well-studied classes of high-dimensional distributions satisfy \cref{thm:the-norm-cond}. For example, it is not hard to see that the conditions of \cref{prop:subexponential-Gaussian-norm} can be achieved through measures satisfying the Poincar\'e and log-Sobolev inequalities, thanks to Lipschitz concentration. For background on Poincar\'e and log-Sobolev inequalities we refer to \cite{Ledoux_2001}.

\begin{prop} \label{prop:Poincare}
    Assume the base vector $\mathbf X^{(n)}$ is isotropic, exchangeable, unconditional, and has uniformly bounded Poincar\'e constant. When $d = o(n^{1/3})$, we then have with probability $1$ that \[\ESD(K_{\PrincipalT}) \wkconv \mu_{\MP(c)}.\]
\end{prop}
\begin{prop} \label{prop:log-Sobolev}
    Assume the base vector $\mathbf X^{(n)}$ is isotropic, exchangeable, unconditional, and has uniformly bounded log-Sobolev constant $C$. When $d = o(\sqrt n)$, we then have with probability $1$ that \[\ESD(K_{\PrincipalT}) \wkconv \mu_{\MP(c)}.\]
\end{prop}

\textcite{KLS_1995} famously conjectured that all isotropic log-concave measures\footnote{By this we mean measures absolutely continuous on $\R^n$ with densities that are log-concave.} on $\R^n$ satisfies the Poincar\'e inequality with a universal constant independent of the dimension $n$. Assuming the correctness of the conjecture, it follows from \cref{prop:Poincare} that for log-concave measures, when $d = o(n^{1/3})$ we would have a.s.\ $\ESD(K_{\PrincipalT}) \wkconv \mu_{\MP(c)}$. The truth is that, if the KLS conjecture is indeed resolved, we can in fact reach the optimal $d = o(\sqrt n)$ growth condition by using properties particular to unconditional log-concave measures.

The current best bound in the KLS conjecture is due to Klartag \cite{klartag2023logarihtmic}, who showed that the Poincar\'e constant of any isotropic log-concave measures in $\R^n$ is bounded by $c_1\log n$ for some absolute constant $c_1 > 0$. Using this we prove the following result.

\begin{prop} \label{prop:log-concave}
    Assume the base vector $\mathbf X^{(n)}$ is isotropic, exchangeable, unconditional, and log-concave. Then when $d =o(\sqrt{n/\log n})$, we have $\ESD(K_{\PrincipalT}) \wkconv \mu_{\MP(c)}$ with probability $1$.
\end{prop} 

% We remark that the expected full resolution of the famous conjecture by \textcite{KLS_1995} would improve $d = o(\sqrt{n/\log n})$ to $d = o(\sqrt n)$ for the case of log-concave measures. \fc{move earlier, example 2.8 at best gives us n13}

\subsubsection{Invariance under \texorpdfstring{$\ell^k$}{lk} norm}
The conditions of \cref{thm:main} can be directly verified when the base vector $\mathbf X$ follows an $\ell^k$-spherical distribution. Such distributions were studied in \cite{schechtman1990volume,rachev1991approximate,gupta1997lp,szablowski1998uniform}.
To define these distributions, let $0 < k < \infty$. The (standard) $k$-Gaussian random variable is given by the density \[
    x\mapsto \frac{k}{2k^{1/k}\Gamma(1/k)} \exp\biggl(-\frac{\abs{x}^k}{k}\biggr).
\] Notice that the $k$-Gaussian has zero mean and variance $k^{2/k} \frac{\Gamma(3/k)}{\Gamma(1/k)}$, and that taking $k = 2$ would precisely recover the standard Gaussian. We say a random variable follows an $n$-dimensional (standard) \df{$k$-Gaussian distribution} if it has i.i.d.\ components of standard $k$-Gaussian random variables.

Let $R$ be a nonnegative random variable, and write $\mathbf X = (X_1,\dotsc,X_n) \sim \Inv_k(R)$ if \[
    \mathbf X \eqD R\frac{\mathbf Z}{\nm{\mathbf Z}_k},
\] where $\mathbf Z$ is a $k$-Gaussian independent of $R$. Here we define $\nm{(z_1,\dotsc,z_n)}_k = \bigl(\sum_{i=1}^n \abs{z_i}^k\bigr)^{1/k}$ for $0 < k < 1$, although $\nm{\blank}_k$ would no longer satisfy the triangle inequality. In addition we write $(X_1,\dotsc,X_n) \sim \Inv_\infty(R)$ if \[
    \mathbf X \eqD R\frac{\mathbf Z}{\nm{\mathbf Z}_\infty},
\] where $\mathbf Z$ is the uniform measure on the $[-1,1]^n$ cube. Note that if we take $R = 1$ a.s., then the vector $\mathbf X$ is distributed according to the cone measure on the unit $\ell^k$-sphere. Therefore we call $\Inv_k(R)$ the \df{$\ell^k$ spherical distribution with positive radius $R$}.

A main motivation behind the $\Inv_k(R)$ distributions comes from the case $k=2$. It is well-known that every rotationally invariant random vector $\mathbf X$ satisfying $\Pr(\mathbf X \neq 0) = 1$ can be characterized as the product of a radial component $R = \nm{\mathbf X}_2$ and an independent spherical component $\mathbf s = \frac{\mathbf X}{\nm{\mathbf X}_2} \sim \Unif(S^{n-1})$; see for example \cite[Proposition~7.3]{Eaton_1984} or \cite[Theorem~2.5]{Fang_Kotz_Ng_1990}.

The $\Inv_k(R)$ distributions are exchangeable and unconditional while having computable moments. Going back to our tensor power models, the block moments $\E(X_1^4\dotsm X_r^4 X_{r+1}^2\dotsm X_{2d-r}^2)$ appearing in \cref{thm:main} are easy to compute for $\Inv_k(R)$ distributions, thanks to the independence between $\mathbf Z$ and $R$ in the definition.

\begin{prop} \label{prop:lk-general-cond}
Let $d = o(\sqrt n)$. Assume $(X_1,\dotsc,X_n) \sim \Inv_k(R)$ for some fixed $k \in (0,\infty]$. Define \[Q(k) = \begin{cases} \frac{\Gamma(3/k)}{\Gamma(1/k)} \cdot \bigl(\frac{n}{k}\bigr)^{-2/k} & \text{if }k<\infty, \\ 1/3 & \text{if }k=\infty. \end{cases}\] Then as $n \to \infty$, if \begin{enumerate}
    \item ${Q(k)}^d \cdot \E(R^{2d})\to 1$, and \label{cond:expec-lk}
    \item ${Q(k)}^{2d} \cdot \bigl[\E(R^{4d}) - \bigl(\E R^{2d}\bigr)^2\bigr] \to 0$, \label{cond:var-lk}
\end{enumerate}
then the three conditions in \cref{thm:main} are satisfied, and we have $\ESD(K_{\PrincipalT}) \wkconv \mu_{\MP(c)}$ with probability $1$.
\end{prop}

In particular, take $k = 2$, we directly get the following for rotationally invariant base vectors.

\begin{cor} \label{cor:rot-inv-cond}
    Assume $d = o(\sqrt n)$. Let $\mathbf X\sim \Inv_2(R)$.  If $\frac{1}{n^{d}}\E(R^{2d}) \to 1$ and $\frac{1}{n^{2d}}\Var(R^{2d}) \to 0$, then we have $\ESD(K_{\PrincipalT}) \wkconv \mu_{\MP(c)}$ with probability $1$. % \fc{This is not true for the isotropic multivariate Laplace distribution, because we would have log-convexity.}
\end{cor}

% If we take $k \to \infty$, then $\frac{\Gamma(3/k)}{\Gamma(1/k)} \cdot \bigl(\frac{n}{k}\bigr)^{-2/k} \to 1/3$. The following result is hence expected for $\Inv_\infty$ distributions. \dm{Why do we need is separate from 2.12? Seems like we can include $k=\infty$ in the theorem }
% \begin{prop}\label{prop:linfty-general-cond}
%     Assume $d = o(\sqrt n)$, and $(X_1,\dotsc,X_n) \sim \Inv_\infty(R)$. As $n \to \infty$, if \begin{enumerate}
%     \item $(1/3)^d \E(R^{2d})\to 1$, and \label{cond:expec-linfty}
%     \item $(1/3)^{2d} \bigl[\E(R^{4d}) - \bigl(\E R^{2d}\bigr)^2\bigr] \to 0$, \label{cond:var-linfty}
% \end{enumerate} 
% then the three conditions in \cref{thm:main} are satisfied, and we have $\ESD(K_{\PrincipalT}) \wkconv \mu_{\MP}$ with probability $1$.
% \end{prop}

% \dm{Maybe just mention that verbally instead of stating a specific example for spheres. I.e. by taking $R = c$ for some appropriate constant, the theorem applies to spheres.} \fc{say for sphere, very easy to check, and ball easy to check, so left to the reader as an exercise} And after finding the appropriate radius that makes the uniform measure on the $\ell^k$ spheres and balls isotropic, we confirm the following.

Consider the constant \[R = R(k) = \begin{cases}
  \sqrt{\frac{\Gamma(1/k)}{\Gamma(3/k)}\cdot \frac{\Gamma\bigl(\frac{n+2}{k}\bigr)}{\Gamma\bigl(\frac{n}{k}\bigr)}} & \text{if }k<\infty, \\  \sqrt{\frac{3n}{n+2}} & \text{if }k=\infty.
\end{cases} \] that makes the $\Inv_k(R)$ an isotropic distribution on the $\ell^k$ sphere with radius $R$. For this $R$, it is not hard to verify the correctness of \cref{prop:lk-general-cond}\ref{cond:expec-lk}. Moreover, since $R$ is constant, \cref{cond:var-lk} is automatic. A similar computation holds for the random variable $R$ that makes $\Inv_k(R)$ an isotropic distribution uniform on the $\ell^k$ ball with radius $R$. The details are left to the reader.

\begin{exa} \label{exa:ell-k-sph-ball}
    Assume $d = o(\sqrt n)$. For i.i.d.\ samples of $\PrincipalT(n,d)$ generated from the cone measure on the isotropic $\ell^k$ sphere or from the uniform measure on the isotropic $\ell^k$ ball ($0 < k\leq \infty$), we have with probability $1$ that $\ESD(K_{\PrincipalT}) \wkconv \mu_{\MP(c)}$.
\end{exa}

We mention that the $\Inv_2(\sqrt n) = \Unif(\sqrt{n}S^{n-1})$ case was included in \textcite[Remark~2.6]{Yaskov_2023}, which was the only non-independent base vector he considered. \Cref{exa:ell-k-sph-ball} is thus a generalization of this result to measures on $\ell^k$ spheres and balls for any fixed $k$, even when $k < 1$ and the $\ell^k$ ball is no longer convex.

\section{Further related work}
\subsection{MP law for samples with relaxed independence conditions}
As mentioned in the introduction, there has been a substantial amount of work on the MP law with relaxed independence conditions within the columns and between the columns. The most up-to-date results, with thorough discussions on previous literature, can almost entirely be found in \cite{yaskov2025spectra}. We point out a few previous papers that are related to our results. Samples following certain rotationally invariant distributions \cite{Yin_Krishnaiah_1986} and the cone measure on the $\ell^k$ ball \cite{Aubrun_2006} have been studied early on. Famously \textcite{Pajor_Pastur_2009} proved the convergence to MP when the i.i.d.\ samples are \df{good vectors}, in the sense that the sequence of sample vectors $\mathbf x_p \in \R^p$ should satisfy \begin{equation} \label{eq:good}
    \Var(\mathbf x_p^\trp A_p\mathbf x_p) = o(p^2) \quad \text{for } \opnm{A_p} \leq 1.
\end{equation} This in particular contains isotropic log-concave distributions, and hence uniform distributions on convex bodies such as the $\ell^k$ balls. We also mention that this is true when the i.i.d.\ samples satisfy the Poincar\'e inequality with a parameter that is $o(p)$. If we define $g(\mathbf x) = \mathbf x^\trp A \mathbf x$, then applying the Poincar\'e inequality gives us precisely \[\Var g(\mathbf x) \leq o(p)\E \nm{\grad g(\mathbf x)}_2^2 = o(p) \E\nm{A\mathbf x}^2= o(p^2).\]

Most results that establish general conditions on the sample vector $\mathbf x_p$ require a weak concentration of its quadratic form, similar to \eqref{eq:good}. Examples include the often-cited result by \textcite{Bai_Zhou_2008} and the work by \textcite{Adamczak_2013}. The latter established a sufficient condition \eqref{eq:exchangeable-uncond-sufficient} when the sample vector $\mathbf x_p$ is exchangeable and unconditional, which was discussed in \cref{sec:general-cond-exchange-uncond-base}. Over this line of work, the best sufficient conditions on the sample vectors were presented in \cite[Section~5]{yaskov2025spectra}, and will be discussed soon in the \cref{sec:general-suff}.

We point out that matrices with exchangeable entries have also been studied in the spectral analysis of other random matrix models, given it is the natural extension beyond assuming i.i.d.\ entries. To give a few examples, \cite{circular_exchange_2016} contains the circular law for matrices with exchangeable entries, and \cite{semicircle_exchange_2006} discusses the semicircle law for symmetric matrices with exchangeable entries in the upper triangle.

Given a data matrix $\mathbb X \in \R^{p\times m}$, we may also take the MP law as the limiting ESD of the sample covariance matrix $\frac{1}{m}\mathbb X \mathbb X^\trp$. From this perspective, \textcite{fleermann_heiny_CW} assigned the whole data matrix Curie--Weiss entries, which is exchangeable but not unconditional. To be specific, the entries $(X_1,\dotsc,X_{pm})$ have distribution given by \[
    \Pr(X_1 = x_1,\dotsc, X_{pm} = x_{pm}) = \frac{1}{Z_{\beta,pm}} \exp\biggl(\frac{\beta}{2pm}\Bigl(\sum_{j=1}^{pm} x_j\Bigr)^2\biggr) \quad\text{for } x_1,\dotsc,x_{pm} =\pm 1,
\] where $\beta > 0$ and $Z_{\beta,pm}$ is the normalization constant. They showed when $\beta \leq 1$, the ESD converges to MP, and when $\beta >1$, a scaled ESD converges to MP. % Since function classes  % \fc{in fact, according to that paper, it is probably possible to prove a semicircle law when $p/m \to 0$ for our model as well.}

\subsection{The two random tensor models} \label{sec:two-tensor-models}
Below we specialize to the case where the samples follow two distinct random tensor models. We start from the principal tensor model. It was first proven by \textcite{BVZ_2021} that \cref{thm:indep-tensor} holds if $d = o(n^{1/3})$. A more careful calculation by \textcite{Yaskov_2023} yielded the improved $d = o(\sqrt n)$. After that, \textcite{yaskov2025remark} showed that in the special case where $\mathbf X^{(n)}$ has independent complex components identically distributed as $X$ for all $n$, with $\E X = 0$ and $\abs{X} = 1$ a.s.\ (which contains the special case $\Unif\{-1,1\}^{n}$), we can take $\min\{d,n-d\} = o(n)$.
\textcite{diaconu2026empirical} independently proved $d = o(\sqrt n)$ for base vectors $\mathbf X$ under stronger assumptions, and presented the case for samples distributed according to $\PrincipalT(n,d,\mathbf X)$ multiplied by a deterministic PSD matrix. This allows the samples to have a non-identity population covariance.

Some words need to be said about why $d = o(\sqrt n)$ is optimal in general for base vectors $\mathbf X^{(n)}$ with the same i.i.d.\ components for all $n$, as appeared in \cref{thm:indep-tensor}\ref{enu:non-as-iff}. The intuition comes from the seminal work introducing $U$-statistics in \cite{Hoeffding_1948_Ustat}. Roughly speaking, we may define the $U$-statistic \[U_n = U_n(X_1,\dotsc,X_n) =  \frac{\mathbf x^\trp \mathbf x}{p} = \frac{1}{\binom{n}{d}} \sum_{j \in \binom{[n]}{d}} \prod_{\alpha \in j} X_\alpha^2.\] Given $X_1,\dotsc,X_n$ are i.i.d., \cite[Theorem~5.2]{Hoeffding_1948_Ustat} tells us that 
\[
   \frac{d^2}{n}(\E X_1^4 - 1)  = \frac{d^2}{n}\Var(X_1^2)\leq \Var(U_n).
\] Therefore, if $n = O(d^2)$ and $\Pr(\abs{X_1}= 1) < 1$, then $\liminf_n \Var(U_n) > 0$. Therefore $U_n$ does not converge in $L^2$. Now, \cite[Theorem~2.1]{Yaskov_2016_necessary_sufficient} states that when $\E \mathbf x\mathbf x^\trp = I_p$, a necessary condition for the ESD to converge to the standard MP law is for $U_n \to 1$ in probability.

The above observation suggests that it is hard to improve beyond $d = o(\sqrt n)$. \cite[Theorem~2.3 and Theorem~2.4]{Yaskov_2023} essentially made this idea rigorous and proved a weak law of large numbers for our $U$-statistic. From there they established the ``only if'' direction of \cref{thm:indep-tensor}\ref{enu:non-as-iff}, which proves the optimality of $d = o(\sqrt n)$ in general. We mention that ``if and only if'' cannot be proven for our results in the exchangeable setting, since in general $\E \mathbf x\mathbf x^\trp \neq I_p$. % \fc{consider moving after 2.1} % \fc{I think this is still not good enough. Also I feel like moving to the related work section here is an eligible choice.}

A closely related model is the \df{tensor product model} (sometimes called the \df{nonsymmetric random tensor model} in the literature). Instead of taking the tensor product of a single vector with itself $d$ times, we may also take samples to be tensor products $\mathbf X_1 \otimes \dotsb \otimes \mathbf X_d$ of i.i.d.\ random vectors $\mathbf X_1,\dotsc,\mathbf X_d$ in $\C^n$. The model was first studied in \cite{ambainis2012random}, in light of applications in quantum information theory. \citeauthor{ambainis2012random} considered the model when $\mathbf X_1$ is uniformly distributed on the unit circle in $\C^n$ and when $\mathbf X_1$ is distributed according to the complex standard Gaussian. They showed, using the moment method, the convergence in expectation of the ESD to the MP law. Furthermore, they gave concentration bounds on the largest eigenvalue of the sample covariance matrix around the upper edge of the MP law. This result is significantly improved by \textcite{lytova2018central}, who showed that as long as $d = o(n)$ and $\mathbf X_1$ are good vectors as in \eqref{eq:good}, the ESD weakly converges to the MP law almost surely. The tensor product model was also further studied in \cite{collins2022spectral} under a different scaling $d / n \to \gamma$ for some fixed $\gamma >0$. In this setting, \citeauthor{collins2022spectral} showed that the ESD for the sample covariance matrix converges to a different law, and therefore implies the optimality of $d = o(n)$ in \citeauthor{lytova2018central}'s result. In \cite{yaskov2025remark} and \cite{Yuan_2024}, the case where components of $\mathbf X_1$ are a.s.\ on the complex unit circle are analyzed. Note that the aspect ratio for the tensor product model is $n^d / m$ instead of $\binom{n}{d} / m$ (which is our convention for the tensor power models).

\subsection{Relation to analysis of kernel matrices}
The sample covariance matrix of the tensor power models has appeared in the spectral analysis of dot product kernel matrices, when the sample size $m$ is proportional to $n^d$ (i.e., in the polynomial regime). We introduce an example that shows how our models arise naturally. Consider a kernel function given by the power function $f(y) = y^d$, for some $d \in \N$. Then the kernel matrix $K \in \R^{m \times m}$ defined by \[
    K_{jk} = f\biggl(\frac{\inp{\mathbf X_j}{\mathbf X_k}}{n}\biggr)%\quad \text{or}\quad \tilde f\biggl(\frac{\inp{\mathbf X_j}{\mathbf X_k}}{\sqrt n}\biggr)
\] corresponds to our tensor model. This is because for base vectors $\mathbf X_j,\mathbf X_k$ and their $d$-th tensor powers $\mathbf x^{(j)}$ and $\mathbf x^{(k)}$, we have \[\inp{\mathbf X_j}{\mathbf X_k}^d = \inp{\mathbf X_j^{\otimes d}}{\mathbf X_k^{\otimes d}} = \inp{\mathbf x^{(j)}}{\mathbf x^{(k)}},\] where the first inner product is in $\R^n$, and the second and third inner products are in $\R^{n^d}$. Since $\sum_{k=1}^m \mathbf x^{(k)}{\mathbf x^{(k)}}^\trp \in \R^{p\times p}$ has the same spectrum as the Gram matrix $\bigl[\inp{\mathbf x^{(j)}}{\mathbf x^{(k)}}\bigr]_{j,k} \in \R^{m \times m}$, up to some zero eigenvalues, our results about sample covariance matrices can be translated into results about the kernel matrix for power function kernels. One can therefore expect the MP law to play a role when analyzing polynomial kernel matrices, and even more kernel matrices defined by more general kernel functions.

The study of kernel matrix $K$ given by $K_{jk} = f\bigl(\frac{\inp{\mathbf X_j}{\mathbf X_k}}{n}\bigr)$ has direction applications to analyzing the test errors of kernel ridge regression (KRR). To analyze the spectrum of $K$ for data with distribution $\nu$ uniform on the cube $\{-1,1\}^n$ or on the sphere $\sqrt n S^{n-1}$, we can decompose any $f \in L^2(\nu)$ into linear combinations of orthogonal polynomials $\{q_i : i \in \N_0\}$. To analyze the induced kernel matrices $\bigl[q_i \bigl(\frac{\inp{\mathbf X_j}{\mathbf X_k}}{n}\bigr)\bigr]_{j,k}$, \cite[Theorem~2]{misiakiewicz2022spectrum} established \cref{thm:indep-tensor}\ref{enu:hypercube} for fixed $d$. In addition, \cite[Theorem~2]{misiakiewicz2022spectrum} and \cite[Theorem~1]{xiao2022precise} gave the following result. Fix $d$ and the base vector $\mathbf X \sim \Unif(\sqrt n S^{n-1})$, the evaluation of all spherical harmonics of degree $d$ on $\mathbf{X}$ induces an isotropic random vector of dimension $\binom{n+d-1}{d} - \binom{n+d-3}{d-2}$. For i.i.d.\ samples generated this way, the ESD of the sample covariance matrix converges to MP. This result follows from our result for $\PrincipalT\bigl(n,d,\Unif(\sqrt n S^{n-1})\bigr)$, in fact for $d = o(\sqrt n)$. To see this, we observe that up to an orthogonal transformation, $\binom{n}{d}$ number of spherical harmonics are just the principal tensors. Meanwhile $
    \binom{n+d-1}{d} - \binom{n+d-3}{d-2} - \binom{n}{d} = o\bigl(\binom{n}{d}\bigr)
$, and hence the remaining terms in the spherical harmonics are negligible to the limiting ESD. We refer the reader to \cite[Appendix~B]{xiao2022precise}, or the introduction in \cite{Yaskov_2023}, for more details on this approach. 
%Alternatively, this result also follows by applying our results to $\PrincipalT\bigl(n,d,\Unif(\sqrt n S^{n-1})\bigr)$.%; see \cite[Appendix~B]{xiao2022precise} for details or the introduction in \cite{Yaskov_2023} for a brief account.
We also point out the reduced symmetric tensor model from \cref{sec:ext-tensor-power} were used in \cite{pandit2024universality}. Given the sample size $m \asymp n^2$ and under some different data assumptions, the authors analyzed the limiting ESD of kernel matrices, and used it to study the training and test errors of KRR.

% training kernel ridge regression are analyzed in the regime $n^2 \asymp m$, where the data is assumed to have certain moment conditions that match the standard Gaussian distribution. While analyzing the kernel matrix via a second-order approximation,  \cite{pandit2024universality} resorted to the reduced symmetric random tensor model, as defined in \cref{sec:ext-tensor-power}, in the case $d = 2$. \dm{This is very confusing, because in Section 2.4 we had different names. The entire paragraph is a bit confusing.} \fc{In Yaskov's original paper the three symmetric  tensor models are all part of the main discussion, but I feel like the reduced symmetric model is not as important so I combined it into section 2.4. But mentioning it here becomes awkward. Maybe we can be more brief about this paragraph since Zhu et al.\ only used the model for the case $d = 2$ (second-order term).} \fc{be more brief}

One can also define an alternative kernel matrix $\widetilde K_{jk} = f\bigl(\frac{\inp{\mathbf X_j}{\mathbf X_k}}{\sqrt n}\bigr)$, where we normalize by $\sqrt{n}$ instead of by $n$.\footnote{One may also take $f\bigl(\frac{\sqrt n\inp{\mathbf X_j}{\mathbf X_k}}{\nm{\mathbf X_j}\nm{\mathbf X_k}}\bigr)$, which shares some similar results.} Compared with the kernel matrix $K$, the argument $\frac{\inp{\mathbf X_j}{\mathbf X_k}}{\sqrt n}$ is now typically of order $1$ instead of $o(1)$. Therefore the kernel matrix $\widetilde K$ contains information of the kernel function $f$ not just locally around $0$, but also away from $0$. For $\mathbf X \sim \Unif(\sqrt n S^{n-1})$ and $\mathbf X$ having i.i.d.\ components with finite moments, \cite{lu2025equivalence} and \cite{dubova2023universality} considered the normalized kernel matrix with zero diagonal \begin{equation} \label{eq:zero-diag-ker} A = \begin{cases}
   \frac{1}{\sqrt m}\widetilde K_{jk} &\text{if }j\neq k, \\ 0 & \text{if }j=k.
\end{cases}\end{equation} They showed in the polynomial regime $m \asymp n^d$, the ESD of $A$ is asymptotically a free additive convolution between a shifted MP law and a semicircle law. It is clear from \cite[Section~3]{dubova2023universality} that by writing the kernel function $f$ into Hermite basis, we can reduce the study of $A$ into the study of a large sample covariance matrix of principal tensors, which they approached directly via resolvent analysis. Our results in this paper partially recovers \cite[Theorem~2.1]{dubova2023universality} for power function kernel (as discussed at the beginning of this subsection) with more general non-i.i.d.\ distributions. We note that this problem was previously studied in the linear regime $m \asymp n$ in \cite{Cheng_Singer_2013} and \cite{Do_Vu_2013}.

% \dm{Maybe drop this paragraph?}We also mention \cite{kogan2025extremal}, which considered the extreme eigenvalues of the empirical spectrum for the principal tensor model under a different scaling. In particular, for the kernel matrix $A$ defined in \eqref{eq:zero-diag-ker}, they described the asymptotic behavior of $\opnm{A}$ in relation to edge of the limiting spectrum of $A$ in the polynomial regime. This work extends \cite{fan2019spectral}, which considered $\opnm{A}$ in the linear regime.

% \fc{this section requires polishing and perhaps further background and motivation}

% \fc{Hadamard matrix, something needs to be said earlier or here}

% In particular, Theorem~1.2 extreme eigenvalues % \fc{the problem with the paper, the assumption changed from the arXiv version without explanation, and the kernel function is not given any additional assumption; also in the proof of theorem 1.2, they mentioned using CLT, but it does not make sense there at all}

\section{Preliminaries}
\subsection{General sufficient conditions for MP} \label{sec:general-suff}
The proof of \cref{thm:main} is covered in \cref{sec:concentration-proof-of-main}. It hinges on the following sufficient condition for the MP law when the samples are i.i.d.

\begin{thm}[{\cite[Theorem~2.1]{Yaskov_2016_necessary_sufficient}\cite[Theorem~5.2]{yaskov2025spectra}}] \label{thm:MP}
    For each $p \in \N$, let $\mathbf x^{(1)}_p,\dotsc,\mathbf x^{(m)}_p$ be $m = m(p)$ i.i.d.\ random vectors in $\R^p$ (with no assumption on the mean and covariance matrix). Consider the sample covariance matrix $K_{p} = \frac{1}{m} \sum_{k=1}^m\mathbf{x}_p^{(k)}{\mathbf{x}_p^{(k)}}^\trp$. As $p \to \infty$, suppose the aspect ratio $p/m \to c$ for some fixed $c > 0$, and \begin{equation}\frac{\mathbf x^\trp_p A_p \mathbf x_p - \tr A_p}{p} \to 0 \quad \text{in probability} \label{eq:conv-in-prob}\end{equation} for any sequence of positive semidefinite (PSD) matrices $A_p \in \R^{p\times p}$ with $\opnm{A_p} \leq 1$, then \[
        \ESD(K_p) \wkconv \mu_{\MP(c)}
    \] with probability $1$, where $\mu_{\MP(c)}$ is specified in \eqref{eq:MP}.
     % Here $\mu$ is a deterministic distribution, whose Stieltjes transform $s = s_\mu$ is given by the unique solution in $\{s\in \C: -\frac{1-r}{z} + rs \in \C^+\}$ to the equation \begin{equation} \label{eq:aniso-MP}
    %     s(z) = \int_0^\infty \frac{1}{\lambda\bigl(1 - r - r z s(z)\bigr) - z} \,d\rho(\lambda)
    % \end{equation} for all $z \in \C^+$.
\end{thm}

% The distribution $\mu$ here is called the \df{Marchenko--Pastur law}. In the case where $\ESD(\Sigma_p) \wkconv \delta_c$, the Dirac point mass at some $c > 0$, then \eqref{eq:aniso-MP} becomes \[
%     s(z) = \frac{1}{c\bigl(1 - r - r z s(z)\bigr) - z}.
% \] This can then be converted into the quadratic equation \[
%     crz[s(z)]^2 + (z + cr - c)s(z) + 1 = 0.
% \]
% Because Stieltjes transform can only take value in $\C^+$, the unique solution can be written out explicitly: \[s(z) = \frac{(1 - r)c - z + \sqrt{[z - (1+r)c]^2 - 4rc^2}}{2rcz}.\] Here the square root of the complex number is taken to be the one with positive imaginary part \dm{Somewhere maybe it would be helpful to say a sentence about the Stieltjes transform}. By \cite[Lemma~3.11]{Bai_Silverstein_2010}, this can then be identified with the measure where $a = c(1 - \sqrt{r})^2$ and $b = c(1 + \sqrt r)^2$, and $(1-\frac{1}{r})^+ = \max\{0,1-\frac{1}{r}\}$.
% We will call this the $c$-MP law, and when $c = 1$, we will just write $\mu_{\mathrm{MP}}$ and call it \emph{the} MP law. The aspect ratio $r$ is usually implicit, and hence omitted.

The most significant part of the above theorem is that there is no moment assumption on the vectors except for the concentration in probability of the quadratic form in \eqref{eq:conv-in-prob}.

% Here is a necessary condition for the ESD of the sample covariance matrix to converge to the Marchenko--Pastur law.
% \begin{thm} \label{thm:necessary-MP}
%     Assume $\E \mathbf x\mathbf x^\trp = I_p$, and for the sample covariance matrix $K_p = \frac{1}{m}\sum_{k=1}^m\mathbf{x}^{(k)}{\mathbf{x}^{(k)}}^\trp$, $\ESD(K_p) \wkconv \mu_{\mathrm{MP}}$ with probability $1$. Then it must hold that $\frac{\mathbf x^\trp \mathbf x}{p} \to 1$ in probability.
% \end{thm}

More generally, we can assign each of the $m$ samples a random weight $R_k$. Note that the theorem below is not the exact same statement as \cite[Theorem~5.14]{yaskov2025spectra}, but by a more straightforward application of \cite[Theorem~4.2]{yaskov2025spectra}, one can obtain the following. Be aware that we now require assumptions on the population covariance.

\begin{thm} \label{thm:weighted_MP}
    For each $p \in \N$, let $\mathbf x^{(1)}_p,\dotsc,\mathbf x^{(m)}_p$ be $m = m(p)$ i.i.d.\ random vectors in $\R^p$, with $\Sigma_p = \E \mathbf x\mathbf x^\trp$ satisfying $\tr \Sigma_p^2= o(p^2)$, and also \begin{equation} \label{eq:anisotropic-WCP}
    \frac{\mathbf x_p^\trp A_p \mathbf x_p - \tr(\Sigma_pA_p)}{p} \to 0\quad\text{in probability}\end{equation} for any sequence of PSD matrices $A_p \in \R^{p\times p}$ with $\opnm{A_p} \leq 1$. Assume $\ESD(\Sigma_p) \wkconv \rho$ for some deterministic measure $\rho$. Assume in addition for each $p$ and $m = m(p)$, we have a diagonal matrix $T_p = \diag\{R_1,\dotsc,R_{m}\}$, with nonnegative entries and independent of $\bigl\{\mathbf x_p^{(k)}\bigr\}_{k=1}^m$, and $\ESD(T_p) \wkconv \tau$ for some deterministic measure $\tau$ with probability $1$. Then, assuming $p / m \to c$ for some fixed $c > 0$, the weighted sample covariance matrix $\widetilde K_{p} = \frac{1}{m} \sum_{k=1}^m R_k \mathbf{x}_p^{(k)} {\mathbf{x}_p^{(k)}}^\trp$ converges weakly with probability $1$ to some deterministic measure $\mu$, which is uniquely determined by its Stieltjes transform \[
        s_\mu(z) = \int_0^\infty \frac{1}{\lambda \int_0^\infty \frac{x}{1 + c x \tilde s(z)}\,d\tau(x) - z }\,d\rho(\lambda)\quad \text{for }z\in \C^+,
    \] where $\tilde s$ is defined by \[
        \tilde s(z) = \int_0^\infty \frac{\lambda}{\lambda \int_0^\infty \frac{x}{1 + c x \tilde s(z)}\,d\tau(x) - z }\,d\rho(\lambda) \quad \text{for }z\in \C^+.
    \]
\end{thm}

We remark that very often in applications, it is easier to prove $\Var(\mathbf x^\trp A\mathbf x) = o(p^2)$ for $\opnm{A_p} \leq 1$. Chebyshev's inequality then implies \eqref{eq:anisotropic-WCP}. When $\mathbf x$ is isotropic, $\tr(\Sigma A) = \tr A$, and therefore we recover condition \eqref{eq:conv-in-prob} in \cref{thm:MP}. Indeed, proving $\Var(\mathbf x^\trp A\mathbf x) = o(p^2)$ was how \cite{Yaskov_2023} (and \cite{BVZ_2021}) approached \cref{thm:indep-tensor} for the base vectors having independent components. We adapt this argument to any exchangeable and unconditional base vector.

Returning to the most important case $\rho = \delta_1$ in \cref{thm:weighted_MP}, the equation for $s_\mu$ simplifies to \[
    s_\mu (z) = \frac{1}{\int_0^\infty \frac{x}{1 + cx s_\mu(z)}\,d\tau(x) - z}.
\]
We are interested in the case where $R_1,\dotsc,R_m$ are i.i.d.\ according to some measure $\tau$, so that the distributions between columns remain independent. Then with probability $1$, the empirical distribution of $R_1,\dotsc,R_m$ (or equivalently, the ESD of $T_p$) must converge weakly to $\tau$ as $p \to \infty$. (This is a consequence of the Glivenko--Cantelli theorem.) 

The proof of \cref{thm:weighted-sample} now follows from \cref{thm:indep-tensor} and \cref{thm:main}.

\begin{proof}[Proof of \cref{thm:weighted-sample}]
    For the population covariance matrix $\Sigma_p = \E \mathbf x\mathbf x^\trp$, we need to check that $\tr \Sigma_p^2 = o(p^2)$, $\ESD(\Sigma_p)\wkconv \delta_1$ with probability $1$, and condition \eqref{eq:anisotropic-WCP}.
    
    First, $\Sigma_p$ is always zero off the diagonal: For indices $i \neq j$ in $\binom{[n]}{d}$, we always have $\E(x_i x_j) = 0$, since there is at least one $X_\alpha$ with $\alpha \in i - j$. We can then use the mean zero independence assumption, or the unconditional assumption.
    
    Meanwhile, on the diagonal of $\Sigma_p$ each entry is $\E(x_1^2) = \E(X_1^2)\dotsm \E(X_d^2) = 1$ by the independence and isotropicity assumption; alternatively by the exchangeability assumption in \cref{thm:main}, $\E(x_i^2) = \E(X_1^2\dotsm X_d^2) \to 1$. Now $\Sigma_p = \E(x_i^2) I_p$, which implies $\ESD(\Sigma_p) \wkconv \delta_1$ and $\tr \Sigma_p^2 = O(p) = o(p^2)$. 
    
    Finally, condition \eqref{eq:anisotropic-WCP} holds. In the proof of \cref{thm:indep-tensor} in \cite{Yaskov_2023}, it was verified that $\Var(\mathbf x^\trp A\mathbf x) = o(p^2)$. We will show this as well in the proof of \cref{thm:main}. As commented before the proof, this directly implies \eqref{eq:anisotropic-WCP}.
\end{proof}

\subsection{Properties of exchangeable distributions}
For product of powers of a finite collection of exchangeable random variables, we have the following consequence of Muirhead's inequality.

\begin{lem}[{\cite[G.2.h]{Marshall_Olkin_Arnolad_2011}}] \label{lem:exchangeable-Schur-cvx}
    If $Y_1,\dotsc,Y_n$ are exchangeable and nonnegative, then \[
        (\alpha_1,\dotsc,\alpha_n) \mapsto \E(Y_1^{\alpha_1}\dotsm Y_n^{\alpha_n})
    \] is Schur-convex over nonnegative exponents  $(\alpha_1,\dotsc,\alpha_n)$ such that the expectation is finite.
\end{lem}
In the context of \cref{thm:main}, since $X_1^2,\dotsc,X_n^2$ are exchangeable and nonnegative, we have for any $1 \leq r \leq d$, 
    \begin{equation}
        \E(X_1^2\dotsm X_{2d}^2) \leq \E(X_1^4\dotsm X_r^4 X_{r+1}^2\dotsm X_{2d - r}^2) \leq \E(X_1^4\dotsm X_{d}^4). \label{eq:moment-cvx}
    \end{equation}
In addition, we also have the inequality \[
    \E(X_1^2\dotsm X_{2d}^2) \leq \frac{1}{n^{2d}} \E\nm{\mathbf X}_2^{4d},
\] which can be seen by expanding the right-hand side. This inequality will become important when relaxing the three conditions of \cref{thm:main}.

Proving \cref{thm:finite-case} also relies on approximating the exchangeable measure $\nu$ by a mixture of product measures, as covered in \cref{sec:finite-d}. The idea comes from the following theorem by \textcite{Diaconis_Freedman_1980}. Let $\mathcal P(\R)$ denote the space of Borel probability measures on $\R$.
\begin{thm}[{\cite[Theorem~13]{Diaconis_Freedman_1980}}] \label{thm:DF}
    Let $\nu$ be an exchangeable distribution on $\R^n$. Then there exists a probability distribution $w$ on $\mathcal P(\R)$ such that for any $k \leq n$,\[
        \bigl\lVert{(p_k)_*\nu - \otimes^k \theta \,dw(\theta)}\bigr\lVert \leq \frac{k(k-1)}{n};
    \] Here $(p_k)_* \nu$ is the projection of $\nu$ onto the first $k$ coordinates, and $\nm{\mu_1 -\mu_2} = 2 \sup_A \mu_1(A) - \mu_2(A)$ over all measurable subsets $A$.
\end{thm}

\cref{thm:DF} essentially says that when an exchangeable measure is projected to a subspace of sufficiently low dimension $o(\sqrt n)$, the projection can be well-approximated by a mixture of product measures in total variation. One should see $o(\sqrt n)$ as the correct regime when working with exchangeability, which we will see again in the proof of \cref{thm:main}. Since we are interested in the block second moments $\E(X_1^2\dotsm X_d^2)$ (and also block fourth moments) of exchangeable measures, we have to appropriately adjust their result; see \cref{lem:exchange-2nd-mmt-approx}. This then allows us to write $\E(X_1^2\dotsm X_d^2) = \int x_1^2\dotsm x_d^2\,d\nu$ into the $2d$-th moment of the mixture measure, which significantly simplifies the three conditions of \cref{thm:main}.

\subsection{Combinatorial estimates} The following two identities will be crucial to our proof of \cref{thm:main}, but will also appear in other places.
\begin{lem} \label{lem:sqrt-n-lim}
    Let $L = L_n$ and $d = d_n$ be sequences of numbers, and we assume $d \geq 1$. Provided that $Ld^2/n\to 0$, we have \begin{equation}
        \lim_{n \to \infty} \biggl(1 + \frac{Ld}{n}\biggr)^{d} = 1. \label{eq:Ldn-growth}
    \end{equation} It follows that when $d = o(\sqrt n)$, $\frac{n\dotsm (n-d+1)}{n^d} \to 1$, $\binom{n-d}{d} / \binom{n}{d} \to 1$, and $\binom{n}{d} / \binom{n + d - 1}{d} \to 1$.
\end{lem}
\begin{proof}
    Since $d\geq 1$ and $Ld^2/n \to 0$, it must be true for large enough $n$ that $1 + \frac{Ld}{n} > \frac{1}{2}> 0$. Therefore $\frac{Ld/n}{1 + (Ld/n)}\leq \log\bigl(1 + \frac{Ld}{n}\bigr) \leq Ld/n$ for any $L$, which implies \[\frac{Ld^2/n}{1 + (Ld/n)}\leq d \log\biggl(1 + \frac{Ld}{n}\biggr) \leq \frac{Ld^2}{n}.\] Taking limits and then exponentiate gives us the desired \eqref{eq:Ldn-growth}.

    Now take $L = -1$, and hence $d = o(\sqrt n)$. 
    By \eqref{eq:Ldn-growth}, we obtain \[\frac{n\dotsm (n-d+1)}{n^d} \geq \biggl(\frac{n-d}{n}\biggr)^d \to 1.\] Since the left-hand side is also bounded above by $1$, it converges to $1$. Similarly it is true that $\binom{n - d}{d}/\binom{n}{d} \to 1$, by observing 
    \[
        \biggl(\frac{n-2d}{n}\biggr)^d\leq \binom{n-d}{d} \bigg/ \binom{n}{d} = \frac{(n-d)\dotsm (n-2d+1)}{n\dotsm (n-d+1)} \leq \biggl(\frac{n-d}{n}\biggr)^d.
    \] The proof of $\binom{n}{d} / \binom{n + d - 1}{d} \to 1$ is similar.
\end{proof}

% We start of with a lemma to be used in the proof of \cref{thm:main}. It already appears in the proof of \cite[Theorem~2]{Yaskov_2023} for tensors induced from independent base vectors, but for the sake of completeness we single this lemma out.
\begin{lem} \label{lem:binom-ineq}
    For $r \leq d \leq n$, the following inequality holds when $d - r \leq n - d$: \[\binom{n - d}{d - r} \leq \biggl(\frac{d}{n}\biggr)^r \binom{n}{d}.\]
\end{lem}
\begin{proof}
    \begin{align*}
        \frac{\binom{n - d}{d - r}}{\binom{n}{d}} = \frac{\frac{(n-d)!}{(n-2d+r)!(d-r)!}}{\frac{n!}{(n-d)!d!}} & = \frac{\frac{d!}{(d-r)!}}{\frac{n!}{(n-r)!}} \cdot\frac{\frac{(n-d)!}{(n-2d+r)!}}{\frac{(n-r)!}{(n-d)!}} \\
        & \leq \frac{(d-r+1)\dotsm d}{(n-r+1)\dotsm n} \cdot \frac{(n-2d+r+1)\dotsm(n-d)}{(n-d+1)\dotsm(n-r)} \\
        & \leq \biggl(\frac{d}{n}\biggr)^r\cdot 1. \qedhere
    \end{align*}
\end{proof}

\iffalse
Our computation in Case~1 also shows that given \cref{cond:2nd-4th-compare}, we cannot relax \cref{cond:var} further. In some sense, this proves the necessity of \cref{cond:var}. (For finite $d$ already, it is unlikely we can drop condition (C) at all, in light of \cref{lem:relax-4th-moment}.)
\begin{prop}
    For each $n$, let $\mathbf X^{(n)} = (X_1,\dotsc,X_n)$ be an exchangeable vector, and let $d = o(\sqrt n)$. Assume \cref{cond:2nd-4th-compare}. As $n \to \infty$, \begin{equation} \label{eq:asymp-pos-corr}
        \liminf_n \E(X_1^2 \dotsm X_{2d}^2) - \bigl(\E X_1^2\dotsm X_d^2\bigr)^2 \geq 0.
    \end{equation} This means precisely that $X_1^2 X_2^2\dotsm X_{d}^2$ and $X_{d+1}^2 X_{d+2}^2\dotsm X_{2d}^2$ cannot be asymptotically negatively correlated.
\end{prop}
\begin{proof}
    Take $A = I_{\binom{n}{d}}$ in Case~1, which has $\Fnm{A}^2 = \binom{n}{d}$. This allows us to drop the AM-GM inequality, so that \begin{align*}
        \Var(\mathbf x^\trp \mathbf x) & = \sum_{i,j} \Cov(\mathbf x_i^2,\mathbf x_j^2) \\
        & \leq M \sum_i \sum_{r=1}^d L^r \sum_j \ind\{\abs{i \cap j} = r\} + \sum_i \sum_j\bigl(M \cdot \ind\{\abs{i \cap j} = 0\} - Q^2\bigr) \\
        & \leq M\binom{n}{d}^2 \biggl[\Bigl(1 + \frac{Ld}{n}\Bigr)^d - 1\biggr] + \binom{n}{d} \biggl[\binom{n - d}{d} M - \binom{n}{d}Q^2\biggr].
    \end{align*}
    Divide both sides by $\binom{n}{d}^2$ and take $\liminf_n$, we must have \[
        0 \leq \liminf_n M  - Q^2,
    \] as desired. Also note that we never used the assumption that $\mathbf X^{(n)}$ is unconditional in Case~1.
\end{proof}
\fi

\section{Concentration of the quadratic form in probability, Theorem~\ref{thm:main}} \label{sec:concentration-proof-of-main}
\begin{proof}[{Proof of \cref{thm:main}}]
    For simplicity, we will write \[Q(n) = \E(X_1^2\dotsm X_d^2) \quad\text{and}\quad M(n) = \E(X_1^2\dotsm X_{2d}^2).\] Our assumptions then become $Q \to 1$, $M -Q^2 \to 0$, and \[
    \E(X_1^4\dotsm X_r^4 X_{r+1}^2\dotsm X_{2d - r}^2) \leq L^r.\] Note since $M \to 1$, for large enough $n$ we always have $M \geq 1/2$. Therefore, replacing $L$ by $2L$ allows us to assume without loss of generality that \[
        \E(X_1^4\dotsm X_r^4 X_{r+1}^2\dotsm X_{2d - r}^2) \leq L^rM.
    \] By \cref{thm:MP}, we claim it is sufficient to check that $\Var(\mathbf x^\trp A\mathbf{x}) = o(p^2)$. As mentioned before, this would imply \eqref{eq:anisotropic-WCP} by Chebyshev's inequality. Therefore to check \eqref{eq:conv-in-prob}, it remains to show \[
        \frac{\tr(\Sigma_p A_p) - \tr A_p}{p} = \frac{(\E x_i^2 - 1)\tr A_p}{p} \to 0 \quad\text{in probability}.
    \] This is correct thanks to $\tr A_p \leq p\opnm{A_p} \leq p$.
    
    We will divide $A = A_{\mathrm{diag}} + A_{\mathrm{off}}$, where $A_{\mathrm{diag}}$ is the diagonal part of $A$. Since \[\Var(\mathbf x^\trp A \mathbf x) \leq 2\Var(\mathbf x^\trp A_{\mathrm{diag}} \mathbf x ) + 2{\Var(\mathbf x^\trp A_{\mathrm{off}} \mathbf x)},\] 
    % This is because for any random variables $Y_1$ and $Y_2$, we have \begin{align*}
    %     \Var(Y_1 + Y_2) & = \Var(Y_1) + \Var(Y_2) + 2\Cov(Y_1,Y_2) \\
    %     & \leq \Var(Y_1) + \Var(Y_2) + 2 \sqrt{\Var(Y_1)\Var(Y_2)}.
    % \end{align*}
    it suffices to show that $\Var(\mathbf x^\trp A_{\mathrm{diag}} \mathbf x ) = o(p^2)$ and $\Var(\mathbf x^\trp A_{\mathrm{off}} \mathbf x ) = o(p^2)$.

    We use $i,j,k,\ell$ for indices in $\binom{[n]}{d}$, the $d$ element subsets of $[n]$. % \fc{talk about this earlier multiindices, notation subsection in the intro} % and $\alpha,\beta,\gamma,\delta$ for indices in $[n]$. 

    \emph{Case 1.} Suppose $A$ is diagonal. Then $\Var(\mathbf x^\trp A\mathbf x) = \sum_{i,j\in \binom{[n]}{d}} a_{ii}a_{jj} \Cov( x_i^2, x_j^2)$. For any two set of indices $i$ and $j$ with $r = \abs{i \cap j}$, by expanding the covariance and using exchangeability, we get \begin{align*}
        \Cov(x_{i}^2, x_j^2) & = \E( x_i^2 x_j^2) - \bigl(\E x_i^2\bigr)^2\\
        & = \E(X_1^4\dotsm X_r^4 X_{r+1}^2\dotsm X_{2d - r}^2) - \bigl(\E X_1^2\dotsm X_{d}^2\bigr)^2, \\
        & \leq L^r M - Q^2.
    \end{align*}
    We can rewrite the above into \begin{align*}
        \bigl\vert{\Cov(x_{i}^2, x_j^2)}\bigr\vert & \leq \biggl\vert\sum_{r = 0}^d \ind\{\abs{i \cap j} = r\} {L^r M - Q^2}\biggr\vert \\
        & \leq \biggl(\sum_{r = 1}^d L^r M \cdot\ind\{\abs{i \cap j} = r\} \biggr) + \bigl\vert M \cdot \ind\{\abs{i \cap j} = 0\} - Q^2\bigr\vert. % \\
        % & \leq \sum_{r = 1}^d 3^r M \cdot\ind\{\abs{i \cap j} = r\}
    \end{align*}% by our assumption.
    % \dm{should $Cov(x_i,x_j)$ be $Cov(x_i^2,x_j^2)$ below? Need to say something if the covariance is negative. FC: Fixed. Already covered in Proposition~5.} % Let us assume in addition $\Cov(\mathbf x_{i}^2,\mathbf x_j^2) \geq 0$, which will be justified at the end.
    Then by AM-GM inequality and symmetry, \begin{align*}
         \Var(\mathbf x^\trp A\mathbf x) & \leq \frac{1}{2} \sum_{i,j} (a_{ii}^2 + a_{jj}^2) \bigl\vert\Cov( x_i^2,  x_j^2)\bigr\vert \\
         & \leq \sum_{i,j} a_{ii}^2 \biggl[\biggl(\sum_{r = 1}^d L^r M \cdot\ind\{\abs{i \cap j} = r\} \biggr) + \bigl\vert M \cdot \ind\{\abs{i \cap j} = 0\} - Q^2\bigr\vert \biggr] \\
         & = M \sum_{i} a_{ii}^2 \sum_{r = 1}^d L^r \sum_j \ind\{\abs{i \cap j} = r\} + \sum_i a_{ii}^2 \sum_{j} \bigl\vert M \cdot \ind\{\abs{i \cap j} = 0\} - Q^2\bigr\vert\\
         & = M \sum_{i} a_{ii}^2 \sum_{r = 1}^d L^r \sum_j \ind\{\abs{i \cap j} = r\} + \sum_i a_{ii}^2 \sum_{j : j\cap i = \emptyset} \abs{M  - Q^2} + \sum_i a_{ii}^2\sum_{j : j\cap i \neq \emptyset} Q^2
         % \\ & = M \Fnm{A}^2 \sum_{r = 1}^d L^r \sum_j \ind\{\abs{i \cap j} = r\} \\ & \qquad + \Fnm{A}^2 \binom{n - d}{d} \abs{M - Q^2} + \Fnm{A}^2\biggl[\binom{n}{d} - \binom{n - d}{d}\biggr] Q^2
     \end{align*} The sum $\sum_j \ind\{\abs{i \cap j} = r\}$ is precisely $\binom{d}{r} \binom{n - d}{d - r}$: fix the choice for $i$, there are $\binom{d}{r}$ choices for $j \cap i$ and $\binom{n - d}{d - r}$ choices for $j - j\cap i $. Therefore \[
        \Var(\mathbf x^\trp A\mathbf x) \leq M \Fnm{A}^2 \sum_{r = 1}^d L^r \binom{d}{r} \binom{n - d}{d - r} + \Fnm{A}^2\binom{n - d}{d} \abs{M - Q^2} + \Fnm{A}^2\biggl[\binom{n}{d} - \binom{n - d}{d}\biggr] Q^2,
     \] which by \cref{lem:binom-ineq,lem:sqrt-n-lim} simplifies to \begin{align}
        \frac{1}{\Fnm{A}^2 \binom{n}{d}}\Var(\mathbf x^\trp A\mathbf x)& \leq M  \sum_{r=1}^d \binom{d}{r} \biggl(\frac{Ld}{n}\biggr)^r +  \frac{\binom{n-d}{d}}{\binom{n}{d}}\abs{M - Q^2} + \biggl[1 - \frac{\binom{n-d}{d}}{\binom{n}{d}}\biggr] Q^2 \nonumber
        \\ & \leq M\biggl[\Bigl(1 + \frac{Ld}{n}\Bigr)^d - 1\biggr]  + \abs{M - Q^2} + o(1)Q^2. \label{eq:diag-var-bound-above}
     \end{align}
    Now by $M - Q^2 \to 0$, \cref{lem:sqrt-n-lim}, and $M$ and $Q^2$ both being bounded, the three terms in \eqref{eq:diag-var-bound-above} all go to $0$. Since $\Fnm{A}^2 \leq p\opnm{A} \leq p = \binom{n}{d}$, we conclude that $\Var(\mathbf x^\trp A\mathbf x) = o(p^2)$.
    
    \emph{Case 2.} Suppose the diagonal of $A$ consists only of zeros. Define $\Delta_r = \sum_{i,j} a_{ij} x_i x_j \ind\{\abs{i \cap j} = r\}$. Then $\E \mathbf x^\trp A\mathbf x = \tr (A\E \mathbf x\mathbf x^\trp) = \tr A = 0$. Therefore \begin{equation}\sqrt{\Var(\mathbf x^\trp A \mathbf x)} = \bigl(\E\abs{\mathbf x^\trp A \mathbf x}^2\bigr)^{1/2} = \nm{\mathbf x^\trp A\mathbf x}_{L^2(\Pr)} \leq \sum_{r = 0}^{d-1}\nm{\Delta_r}_{L^2(\Pr)} = \sum_{r = 0}^{d-1} \bigl(\E \abs{\Delta_r}^2\bigr)^{1/2}, \label{eq:off-diagonal-var-prelim-bound}\end{equation} where \begin{align} 
        \E \abs{\Delta_r}^2 & = \sum_{\abs{i \cap j} = \abs{k \cap \ell} = r} a_{ij}a_{k\ell} \E(x_i x_j x_k x_\ell) \nonumber \\ &  \leq \frac{1}{2}\sum_{\abs{i \cap j} = \abs{k \cap \ell} = r} (a_{ij}^2 + a_{k\ell}^2) \abs{\E x_i x_j x_k x_\ell} \nonumber \\
        & = \sum_{\abs{i \cap j} = \abs{k \cap \ell} = r} a_{ij}^2 \abs{\E x_i x_j x_k x_\ell}.\label{eq:expand-Delta}
        \end{align}  For each tuple $(i,j,k,\ell)$, we can rewrite $\E(x_i x_j x_k x_\ell)$ as \[
            \E(X_1^4\dotsm X_s^4 X_{s+1}^2\dotsm X_{s+t}^2 X_{s+t+1}^3\dotsm X_{s+t+u}^{3}X_{s+t+u+1}\dotsm X_{s+t+u+v})
        \] by exchangeability. Namely, $s,t,u,v$ are respectively the number of indices repeated in the multi-indices $i,j,k,\ell$ for $4,2,3,1$ times. Because $\mathbf X$ is unconditional, if $u \neq 0$ or $v \neq 0$, then the expectation is automatically $0$. Therefore, we only have to focus on the case where \[
            \ \abs{i \cap j} = \abs{k \cap \ell} = r\quad \text{and}\quad 4s + 2t = 4d.\] 
        In this case, \[
                \E (x_{i}x_{j}x_kx_\ell) = \E(X_1^4\dotsm X_s^4 X_{s+1}^2\dotsm X_{s+t}^2) \leq L^s M.
            \]
        This leads us to define for $0 \leq s \leq r \leq d - 1$ that \[\Gamma(s,r) = \bigl\{(i,j,k,\ell) : \abs{i \cap j}=\abs{k \cap \ell} = r \text{ and } 4s + 2t = 4d \bigr\},\] where $s$ and $t$ are as above.
        Our aim is to reduce the sum in \eqref{eq:expand-Delta} to a sum over $i,j$ such that $\abs{i \cap j} = r$, so that we can use an analogous symmetry argument as in Case 1. Suppose $i,j$ have already been chosen so that $\abs{i \cap j} = r$. For each fixed $0 \leq s \leq r$, there are in total $\binom{r}{s}$ choices for $i \cap j \cap k \cap \ell$. The remaining choices for $k \cap \ell - i\cap j \cap k\cap \ell$ are $\binom{n - r}{r - s}$, since it is not allowed to intersect $i \cap j$. Now within $i$ and $j$, there are $2d - 2r$ indices not repeated. We have $\binom{2d - 2r}{d - r}$ choices for $k - k \cap \ell$, and the remaining $d - r$ indices are given to $\ell - k\cap \ell$. Therefore in total there are \[
            \binom{r}{s}\binom{n-r}{r-s}\binom{2d - 2r}{d-r} \quad\text{choices for $k$ and $\ell$},
        \] when the choices of $(i,j), r,s$ are fixed.
        
        Returning back to \eqref{eq:expand-Delta}, we can further bound \begin{align*}
            \E \abs{\Delta_r}^2 & \leq \sum_{s = 0}^r L^s M\sum_{(i,j,k,l) \in \Gamma(s,r)} a_{ij}^2 \\
            &= M \sum_{s = 0}^r L^s\sum_{\abs{i \cap j} = r} a_{ij}^2 \binom{r}{s}\binom{n-r}{r-s}\binom{2d - 2r}{d-r}  \\ 
            & \leq M\biggl(\sum_{\abs{i\cap j} = r} a_{ij}^2\biggr)\binom{2d - 2r}{d-r} \sum_{s=0}^r  L^s  \binom{n}{r}\binom{r}{s}\biggl(\frac{r}{n}\biggr)^s,
        \end{align*} where the last line appeals to \cref{lem:binom-ineq}.
        Therefore continuing from \eqref{eq:off-diagonal-var-prelim-bound}, \begin{align}
       \sqrt{\Var(\mathbf x^\trp A \mathbf x)} \leq \sum_{r = 0}^{d-1} \bigl(\E \abs{\Delta_r}^2\bigr)^{1/2} & \leq \sqrt{M}\sum_{r = 0}^{d-1} \biggl(\sum_{\abs{i\cap j} = r} a_{ij}^2\biggr)^{1/2} \biggl[\binom{2d - 2r}{d-r} \sum_{s=0}^r   \binom{n}{r}\binom{r}{s}\biggl(\frac{Lr}{n}\biggr)^s\biggr]^{1/2} \nonumber \\ 
        & \leq \sqrt{M}\sqrt{\sum_{r = 0}^{d-1} \sum_{\abs{i \cap j} = r} a_{ij}^2} \sqrt{\sum_{r=0}^{d-1} \binom{2d - 2r}{d-r}\binom{n}{r} \sum_{s=0}^r  \binom{r}{s}\biggl(\frac{Lr}{n}\biggr)^s} \nonumber \\ & = \sqrt{M}\sqrt{\Fnm{A}^2}\sqrt{\sum_{r=0}^{d-1} \binom{2d - 2r}{d-r} \binom{n}{r}\biggl(1 + \frac{Lr}{n}\biggr)^r}, \label{eq:case-2-intermed-bound}\end{align}
        where we have used the Cauchy--Schwarz inequality in the second line. Now $\binom{2d - 2r}{d-r}  \leq 4^{d-r}$, and also \begin{align*}
            \frac{\binom{n}{r}}{\binom{n}{d} } =  \frac{\binom{d}{r}}{\binom{n-r}{d-r}}   =  \frac{(r+1)\dotsm d}{(n-d+1)\dotsm (n-r)}  \leq \biggl(\frac{d}{n-d}\biggr)^{d-r} \leq \biggl(\frac{2d}{n}\biggr)^{d-r}
        \end{align*}
        provided that $n \geq 2d$. Therefore \begin{align}
            \sum_{r=0}^{d-1} \binom{2d - 2r}{d-r} \binom{n}{r}\biggl(1 + \frac{Lr}{n}\biggr)^r & \leq \binom{n}{d}\sum_{r=0}^{d-1} \biggl( \frac{8d}{n}\biggr)^{d-r}\biggl(1 + \frac{Lr}{n}\biggr)^r \nonumber\\
            & \leq \binom{n}{d} \biggl( \frac{8d}{n}\biggr)^d\sum_{r=0}^{d-1} \biggl( \frac{8d}{n}\biggr)^{-r}\biggl(1 + \frac{Lr}{n}\biggr)^r. \label{eq:reduce-geo-series}
        \end{align}
    The last sum can be simplified to \begin{align*}
     \sum_{r=0}^{d-1}\biggl(\frac{n + Lr}{8d}\biggr)^r \leq \sum_{r=0}^{d-1} \biggl(\frac{n}{8d} + \frac{L}{8}\biggr)^{r} = \frac{(\frac{n}{8d} + \frac{L}{8})^d - 1}{\frac{n}{8d} + \frac{L}{8} - 1} \leq \frac{(\frac{n}{8d} + \frac{L}{8})^d - 1}{\frac{n}{16d}},
    \end{align*} provided that $\frac{n}{8d} + \frac{L}{8} - 1\geq \frac{n}{16d}$. To make this true, without loss of generality we may assume $n \geq 16d$. Continuing from \eqref{eq:reduce-geo-series}, we obtain \begin{align*}
          \sum_{r=0}^{d-1} \binom{2d - 2r}{d-r} \binom{n}{r}\biggl(1 + \frac{Lr}{n}\biggr)^r & \leq \binom{n}{d} \frac{16d}{n}\biggl(\frac{8d}{n}\biggr)^d \biggl[\Bigl(\frac{n}{8d} + \frac{L}{8}\Bigr)^d - 1\biggr] \\
          & =  \binom{n}{d} \frac{16d}{n} \biggl[\Bigl(1 + \frac{Ld}{n}\Bigr)^d - \Bigl(\frac{8d}{n}\Bigr)^d\biggr].
    \end{align*}
    Now returning back to \eqref{eq:case-2-intermed-bound} gives \[
        \Var(\mathbf x^\trp A\mathbf x) \leq M \Fnm{A}^2\binom{n}{d} \frac{16d}{n} \biggl(1 + \frac{Ld}{n}\biggr)^d,
    \] provided that $n \geq 16d$.
    
    As in the diagonal case, by \cref{lem:sqrt-n-lim} and $M \to 1$, we can conclude that $\Var(\mathbf x^\trp A\mathbf x) = o(p^2)$.
    % It remains to show \cref{lem:relax-4th-moment}, and then our proof will be complete.
\end{proof}

\section{Theorem~\ref{thm:finite-case} for fixed \texorpdfstring{$d$}{d}} \label{sec:finite-d}
In this section, we treat the case when the degree $d$ remains fixed as $n$ goes to infinity and prove \cref{thm:finite-case}.
To motivate the proof, and as a warm-up, we begin with proving the example in \cref{prop:mixture} for mixtures of i.i.d.\ distributions.

\begin{proof}[Proof of \cref{prop:mixture}]
For the base vector we assume \[\mathbf{X} =(X_1,\dotsc,X_n) \sim \otimes^n \theta \,dw_n(\theta).\] Here $\theta$ has unit variance and fourth moment bounded by $L$, and $w_n$ is a measure on $\mathcal P(\R)$. % We write $(X_1;\theta),\dotsc,(X_n;\theta)$ for  be independent and distributed according to some $\theta$, which is an symmetric measure with unit variance and bounded fourth moment.
  To distinguish from the standard expectation $\E$, we shall use $E_{w_n}$ for $\int_{\mathcal P(\R)} \blank\,dw_n(\theta)$. 
    With this notation, we can write \[
        \E(X_1^4\dotsm X_r^4 X_{r+1}^2\dotsm X_{2d - r}^2) = E_{w_n} \biggl(\int x_1^4\,d\theta\biggr)^r \biggl(\int x_1^2\,d\theta\biggr)^{2d - 2r} \leq L^r.
    \] Meanwhile by $\int x_1^2\,d\theta = 1$ for $w_n$-a.e.\ $\theta$, \[
        E_{w_n}\biggl(\int x_1^2\,d\theta\biggr)^d \to 1
    \text{ and }
        E_{w_n}\biggl(\int x_1^2\,d\theta\biggr)^{2d}  \to 1.
    \] Therefore the conditions of \cref{thm:main} are all met.
\end{proof}

The proof motivates the following idea: if we can approximate an exchangeable distribution by a mixture of product distributions, then we might use the above argument to satisfy the conditions of \cref{thm:main}. We start with the following lemma inspired by \cref{thm:DF}.

\begin{lem} \label{lem:exchange-2nd-mmt-approx}
    For each $n$, let $\nu = \nu^{(n)}$ be an exchangeable distribution on $\R^n$. There is an explicit sequence of probability distributions $\{w_n\}$ on $\mathcal P(\R)$ such that \begin{equation}
        \int_{\mathcal P(\R)} \Bigl(\int_{\R} x_1^2 \,d\theta\Bigr)^k \,dw_n = \frac{1}{n^k} \int \bigl(x_1^2 + \dotsm + x_n^2\bigr)^k\,d\nu.
    \end{equation}
    Therefore if for some $k \leq n$, we have \begin{equation}\int x_1^{2}\dotsm x_{k}^2\,d\nu - \frac{1}{n^k} \int \bigl(x_1^2 + \dotsm + x_n^2\bigr)^k\,d\nu \to 0, \label{eq:cond-2nd-mmt-approx}\end{equation}
    then \begin{equation}
        \int x_1^2\dotsm x_k^2 \,d\nu - \int_{\mathcal P(\R)} \Bigl(\int_{\R} x_1^2 \,d\theta\Bigr)^k \,dw_n \to 0. \label{eq:exchange-2nd-mmt-approx}
    \end{equation}
\end{lem}

\begin{proof}
    Fix $n$. For each $t = (t_1,\dotsc,t_n) \in \R^n$, we can consider an urn $U_t = \{t_1,\dotsc,t_n\}$. Drawing one ball from this urn gives us the measure $M_{1,t}$ on $\R$ defined by $M_{1,t}\{x_1\} = \frac{m}{n}$, where $m = \lvert \{i: t_i =x_1\}\rvert$ is the number of times $x_1$ appears in $U_t$. % where $x_j$ represent the outcome of the $j$th draw from the urn.
    % If we are drawing $k$ items from the urn without replacement, then we end up with a measure $H_{k,t}$ on $\R^k$, where $H_{k,t}\{(x_1,\dotsc,x_k)\} = \frac{1}{n\dotsm(n-k+1)}$ for any $(x_1,\dotsc,x_k) \in (U_t)^k$ that are mutually distinct.
    % \dm{The meaning of an urn should be expanded here a bit}. \fc{sampling wo replacement; However, this is only to define a measure by drawing one ball, has nothing to do with replacement}

    Note $t\mapsto M_{1,t}$ is a measurable map from $\R^n$ to $\mathcal P(\R)$, defined independent of $k$. Let $w_n$ be the image measure of $\nu = \nu^{(n)}$ under this map, and let $t_{i_1}$ represent the outcome of this draw. Therefore \begin{align*}
        \int_{\mathcal P(\R)} \Bigl(\int_{\R} x_1^2 \,d\theta\Bigr)^k \,dw_n & = \int_{\R^n} \Bigl(\int_{\R} x_1^2 \,dM_{1,t}\Bigr)^k \,d\nu(t) \\
        & = \int_{\R^n} \biggl(\frac{t_1^2 + \dotsb + t_n^2}{n}\biggr)^k \,d\nu(t) = \frac{1}{n^k} \int \bigl(t_1^2 + \dotsm + t_n^2\bigr)^k \,d\nu(t),
    \end{align*} and thus proving the equivalence between \eqref{eq:cond-2nd-mmt-approx} and \eqref{eq:exchange-2nd-mmt-approx}.
    \end{proof}

    \begin{rem*}
        When $k = 1$, \eqref{eq:cond-2nd-mmt-approx} and \eqref{eq:exchange-2nd-mmt-approx} not only converge to $0$, but are exactly $0$. This explains why we defined $M_{1,t}$ this way. Clearly it is also true that \[
            \int x_1^{2k}\,d\nu = \int_{\mathcal P(\R)} \Bigl(\int_{\R} x_1^{2k} \,d\theta\Bigr) \,dw_n.
        \]
        \iffalse
        When $k = 2$, \[
            \int x_1^2x_2^2 \,d\nu - \frac{1}{n^2} \int \bigl(x_1^2 + \dotsm + x_n^2\bigr)^2\,d\nu = \frac{1}{n} \biggl(\int x_1^2x_2^2\,d\nu - \int x_1^4\,d\nu\biggr).
        \] Assume $\int x_1^4\,d\nu$ is bounded for all $n$, then we may conclude from \cref{lem:exchange-2nd-mmt-approx} that \begin{equation} \label{eq:2nd-mmt-1n-speed}
            \int x_1^2x_2^2\,d\nu - \int_{\mathcal P(\R)} \Bigl(\int x_1^2\,d\theta\Bigr)\,dw_n = O(1/n).
        \end{equation} \fi
    \end{rem*}
The next two lemmas will now allow us to connect \cref{lem:exchange-2nd-mmt-approx} and the conditions in \cref{thm:main}.
\begin{lem} \label{lem:relax-4th-moment}
    For fixed $d$, $\sup_n \E(X_1^4\dotsm X_d^4) < \infty$ is equivalent to \cref{thm:main} \cref{cond:2nd-4th-compare} for a constant $L$: for all $1 \leq r\leq d$, there exists some constant $L$ such that for all large enough $n$, \[
        \E(X_1^4\dotsm X_r^4 X_{r+1}^2\dotsm X_{2d - r}^2) \leq L^r.
    \]
\end{lem}

\begin{proof}
    \Cref{cond:2nd-4th-compare} in particular gives us \[
    \E(X_1^4\dotsm X_d^4) \leq L^d,
\] which proves the uniform boundedness of $\E(X_1^4\dotsm X_d^4)$.

Now suppose $\E(X_1^4\dotsm X_d^4)$ is uniformly bounded, which allows us to always pick some constant $L\geq 1$ such that $\sup_n \E(X_1^4\dotsm X_d^4) \leq L$. Therefore by \cref{lem:exchangeable-Schur-cvx}, for each $d$ and $r$, \[
    \E(X_1^4\dotsm X_r^4 X_{r+1}^2\dotsm X_{2d - r}^2) \leq \E(X_1^4\dotsm X_d^4) \leq L \leq L^r.\qedhere
\]
\end{proof}

\begin{lem} \label{lem:suff-cond-prod-approx}
    Let $k$ be fixed, if $\nu = \nu^{(n)}$ satisfies $\sup_n \int x_1^{2k} \,d\nu < \infty$, then \eqref{eq:cond-2nd-mmt-approx} holds.
\end{lem}
In fact, assuming the stronger $k = o(\sqrt n)$ or $\sup_n \int x_1^{2k} = o(n)$ is also sufficient for \eqref{eq:cond-2nd-mmt-approx}.

\begin{proof}
    Expanding the left-hand side of \eqref{eq:cond-2nd-mmt-approx}, by exchangeability, we obtain \begin{equation}
         \biggl(1- \frac{n\dotsm (n-k+1)}{n^k}\biggr) \int x_1^2\dotsm x_k^2 \,d\nu - \frac{1}{n^k}\sum_{\stackrel{{i_1},\dotsc,{i_k}\in [n]}{\text{not distinct}}} \int x_{i_1}^2 \dotsm x_{i_k}^2\,d\nu \label{eq:mixture-criterion}
    \end{equation} % As $n \to \infty$, under $d = o(\sqrt n)$, the right-hand side would asymptotically be \[
    %     -\frac{1}{n^k} \sum_{\text{some }t_i =t_j} t_1^2\dotsm t_k^2.
    % \] The sum contains $o(n^k)$ terms, and $\int t_1^2\dotsm t_k^2\,d\nu$
    Recall \cref{lem:exchangeable-Schur-cvx} tells us that for any ${i_1},\dotsc,{i_k} \in [n]$, it holds that \[\int_{\R^n} x_{i_1}^2\dotsm x_{i_k}^2\,d\nu \leq \int_{\R^n} x_1^{2k}\,d\nu.\] Therefore the first part in \eqref{eq:mixture-criterion} converges to $0$, since we are multiplying a bounded integral by $1 - \frac{n\dotsm (n-k+1)}{n^k}$, which converges to $0$ by \cref{lem:sqrt-n-lim}.

    Thanks to the two lemmas, a similar argument applies to the second part: \begin{align*}
        \frac{1}{n^k} \sum_{i_1,\dotsc,i_k \text{ not distinct}} \int_{\R^n}t_{i_1}^2\dotsm t_{i_k}^2 \,d\nu(t) & = \biggl(1 - \frac{n\dotsm (n-k+1)}{n^k}\biggr) O(1)
    \end{align*} which converges to $0$.
\end{proof}

% Let $(X_1,\dotsc,X_n)\sim \nu$. For each $n$, let $k = 2d$. Then we can pick an associated measure $w_n$ on $\mathcal P(\R)$ such that \[
%     \bigl\lVert{(p_{2d})_*\nu - \otimes^{2d} \mu_\theta \,dw_n(\theta)}\bigr\lVert \leq \frac{2d(2d-1)}{n},
% \] which goes to $0$ as $n \to \infty$. We claim that for any $k \leq 2d$, it holds as well that \[
%     \bigl\lVert{(p_{k})_*\nu - \otimes^{k} \mu_\theta \,dw_n(\theta)}\bigr\lVert \leq \frac{2d(2d-1)}{n} \to 0.
% \] This is because \[
%     \bigl\lVert{(p_k)_* \mu_1 - (p_k)_* \mu_2} \bigr\rVert = \sup_{B \subseteq \R^{k}} \mu_1\bigl(p_k^{-1} B\bigr) - \mu_2\bigl(p_k^{-1} B\bigr) \leq \sup_{A} \mu_1(A) -\mu_2(A) = \nm{\mu_1 - \mu_2}.
% \]

\begin{proof}[{Proof of \cref{thm:finite-case}}]
By \cref{lem:exchangeable-Schur-cvx}, $\E X_1^{4d}$ is bounded implies $\E(X_1^4\dotsm X_d^4)$ is bounded, which is equivalent to \cref{thm:main} \cref{cond:2nd-4th-compare} for constant $L$. By Jensen's inequality we also know $\E X_1^2$ and $\E X_1^4$ are bounded. Therefore combining \cref{lem:suff-cond-prod-approx} and \cref{lem:exchange-2nd-mmt-approx}, there is a sequence of probability measures $\{w_n\}$ on $\mathcal P(\R)$ such that % \dm{The expectation sign below got switched} \fc{this is deliberate}
\[
    \int x_1^2\,d\nu - E_{w_n} \Bigl(\int_{\R} x_1^2 \,d\theta\Bigr) = 0\text{, } \int x_1^2x_2^2 \,d\nu - E_{w_n} \Bigl(\int_{\R} x_1^2 \,d\theta\Bigr)^{2} \to 0
\] and \begin{equation}
    \int x_1^2\dotsm x_d^2\,d\nu - E_{w_n}\Bigl(\int_{\R} x_1^2 \,d\theta\Bigr)^d \to 0\text{, } \int x_1^2\dotsm x_{2d}^2 \,d\nu - E_{w_n} \Bigl(\int_{\R} x_1^2 \,d\theta\Bigr)^{2d} \to 0. \label{eq:written-out-d-2d-mmt}
\end{equation} Recall $E_{w_n}$ stands for $\int_{\mathcal P(\R)} \blank\,dw_n(\theta)$. Therefore using our assumptions, \[
    E_{w_n} \Bigl(\int x_1^2 \,d\theta\Bigr) \to 1\text{ and } E_{w_n} \Bigl(\int x_1^2 \,d\theta\Bigr)^2 \to 1.
\] % This implies that $w_n\{\theta: \abs{\int x_1^2\,d\theta - 1} > \delta\} \to 0$ for any $\delta > 0$.

We will write $f(\theta) = \int x_1^2\,d\theta$ from now on. Recall each $w_n$ is a probability measure on the Polish space $\mathcal P(\R)$ endowed with the topology of weak convergence. We can then realize each $w_n$ and $\nu^{(n)}$ as the distribution of the random measure $Y_n$ and the random variable $\mathbf X^{(n)}$ on the common probability space $[0,1]$ with the uniform probability measure, which we continue to write as $\Pr$. Hence \[    
    \E f(Y_n) \to 1\text{ and } \E f(Y_n)^2 \to 1.
\] This implies that the random variable $f(Y_n)\to 1$ in probability, and hence for fixed $d$, $f(Y_n)^{2d} \to 1$ in probability as well. Since \begin{align*}
    \sup_n \E \bigl(\abs{f(Y_n)}^{2d + \epsilon/2}\bigr) & = \sup_n E_{w_n} \Bigl(\int_\R \abs{x_1}^2\,d\theta\Bigr)^{2d + \epsilon / 2} \\ & \leq \sup_n E_{w_n} \Bigl(\int_\R \abs{x_1}^{4d+\epsilon}\,d\theta\Bigr) = \sup_n \int \abs{x_1}^{4d+\epsilon}\,d\nu < \infty,
\end{align*} it follows that $f(Y_n)^{2d}$ must be uniformly integrable. Thus, $\E f(Y_n)^{2d} = E_{w_n} \bigl(\int x_1^2 \,d\theta\bigr)^{2d} \to 1$, and also $\E f(Y_n)^{d} = E_{w_n} \bigl(\int x_1^2 \,d\theta\bigr)^{d}\to 1$.

By \eqref{eq:written-out-d-2d-mmt}, it is now immediate that $\int x_1^2\dotsm x_d^2\,d\nu \to 1$ and $\int x_1^2\dotsm x_{2d}^2\,d\nu \to 1$, which recover \cref{cond:expec,cond:var} in \cref{thm:main}.
\end{proof}

\section{Theorem~\ref{thm:the-norm-cond} for diverging \texorpdfstring{$d$}{d}} \label{sec:growing-d}
% \dm{This is a bit out of the blue, should it be part of the next section?}
Previously, in \cref{lem:exchange-2nd-mmt-approx}, we have seen that \eqref{eq:cond-2nd-mmt-approx} was a crucial part of the proof. Now, when $d$ is allowed to depend on $n$, instead of relying on an explicit approximation by a mixture, we establish \eqref{eq:cond-2nd-mmt-approx} directly. We start with a sufficient condition for \eqref{eq:cond-2nd-mmt-approx}.
\begin{lem} \label{lem:2-norm-conv-cond}
     Let $k = k_n$ and $L = L_n$ satisfy $Lk^2 / n \to 0$. Assume that $\E X_1^2 = 1$, $\frac{1}{n^2} \E\nm{\mathbf X}_4^8  \leq L^2$ for all $n$, and also  
    \begin{equation}
        \frac{1}{n^{2k}} \E\nm{\mathbf X}_2^{4k} \to 1. \label{eq:normalized-norm-converge}
    \end{equation} Then, for any $1 \leq r \leq k$, \begin{equation}
        \frac{1}{n^{2r}} \E\nm{\mathbf X}_2^{4r} \to 1. \label{eq:conv-smaller-norm-mmts}
    \end{equation}
    Moreover, we also have \begin{equation} \label{eq:subexponential-approx-norm-full-mix} \E(X_1^2 \dotsm X_k^2) - \frac{1}{n^k} \E\nm{\mathbf X}_2^{2k} \to 0,\end{equation} which implies \[
        \E(X_1^2 \dotsm X_k^2) \to 1.
    \]
\end{lem}

\begin{proof}
We first show \eqref{eq:conv-smaller-norm-mmts}. By Jensen's inequality, \[
        \bigl(\E X_1^2\bigr)^{2r} = \frac{1}{n^{2r}} \bigl(\E\nm{\mathbf X}_2^2\bigr)^{2r}\leq \frac{1}{n^{2r}}\E\nm{\mathbf X}_2^{4r},
    \] where the first equality is due to exchangeability. Therefore, $1 \leq \liminf_n \frac{1}{n^{2r}} \E\nm{\mathbf X}_2^{4r}$. In addition, for any $1\leq r \leq k$, \[
    \biggl\Vert\frac{\nm{\mathbf X}_2^2}{n}\biggr\Vert_{L^{2r}(\Pr)} \leq \biggl\Vert\frac{\nm{\mathbf X}_2^2}{n}\biggr\Vert_{L^{2k}(\Pr)}
\] Hence, since $r \leq k$, we have \begin{equation} \label{eq:non-asymp-bound-r-mmt}
    \frac{1}{n^{2r}} \E\nm{\mathbf X}_2^{4r} \leq \biggl(\frac{1}{n^{2k}} \E\nm{\mathbf X}_2^{4k}\biggr)^{\frac{r}{k}} \leq \max\biggl\{\frac{1}{n^{2k}} \E\nm{\mathbf X}_2^{4k},1\biggr\},
\end{equation} which implies $\limsup_n \frac{1}{n^{2r}} \E\nm{\mathbf X}_2^{4r} \leq 1$.

For the second part of the lemma, we introduce independent uniform random indices $I_1,\dotsc,I_k$, chosen uniformly at random from $[n]$. Let $\mathbf y = (y_1,\dotsc,y_n) \in \R^n$ be a deterministic vector, and define the random variable $M(\mathbf y) = \prod_{j=1}^k y_{I_j}^2$. By independence, \[\E M(\mathbf y) = \bigl(\E y_{I_1}^2\bigr)^k = \frac{1}{n^k}\bigl(y_1^2 + \dotsm +y_n^2\bigr)^k = \frac{1}{n^k} \nm{\mathbf y}_2^{2k}.\] Therefore, for $\mathbf X \sim \nu$ independent of the random indices, \[
    \E M(\mathbf X) = \E\bigl(\E( M(\mathbf X)\giv\mathbf X)\bigr) = \frac{1}{n^k} \E\nm{\mathbf X}_2^{2k}.
\] Let $A = \{\text{for all }j\neq j'\text{, }I_j\neq I_{j'}\}$, then $\E(M(\mathbf X)\giv A) = \E(X_1^2\dotsm X_k^2)$. Recall \cref{lem:exchangeable-Schur-cvx} gives us \[
    \E M(\mathbf X) - \E (M(\mathbf X)\giv A) \geq 0.
\] Therefore, if we can show \[
    \E M(\mathbf X) - \E (M(\mathbf X)\giv A) \leq \E M(\mathbf X) - \Pr(A)\E (M(\mathbf X)\giv A) = \E(M(\mathbf X)\ind_{A^\cpl}) \to 0,\] then \eqref{eq:subexponential-approx-norm-full-mix} follows.
Notice that $\ind_{A^\cpl} \leq \sum_{j < j'} \ind_{I_j = I_{j'}}$, which implies \[\E (M(\mathbf X)\ind_{A^\cpl}) \leq \sum_{j < j'}\E(M(\mathbf X)\ind_{I_j = I_{j'}}).\] Now for each summand, \begin{align*}
	\E(M(\mathbf X)\ind_{I_j = I_{j'}}) &= \E\bigl(\E(M(\mathbf X)\ind_{I_j = I_{j'}}\giv \mathbf X)\bigr)= \Pr(I_j = I_{j'})\E\Bigl(\E\bigl(X^4_{I_j} \bigm\vert \mathbf X\bigr)\E\bigl(\textstyle\prod_{k\notin\{j,j'\}} X^2_{I_k}\bigm\vert \mathbf X\bigr)\Bigr)\\
	&=\frac{1}{n}\E\biggl(\Bigl(\frac{1}{n}\sum X_i^4\Bigr)\Bigl(\frac{1}{n}\sum X_i^2\Bigr)^{k-2}\biggr)\\
	&= \frac{1}{n}\E\biggl(\Bigl(\frac{1}{n}\sum X_i^4\Bigr)\frac{\nm{\mathbf X}_2^{2k-4}}{n^{k-2}}\biggr)\\
	&\leq \frac{1}{n}\sqrt{\frac{1}{n^2}\E\nm{\mathbf X}_4^8}\sqrt{\frac{1}{{n^{2k-4}}}\E\nm{\mathbf X}_2^{4k-8}}.
\end{align*}
Therefore, by our assumption $\frac{1}{n^2}\E\nm{\mathbf X}_4^8 \leq L^2$, we have the upper bound 
\begin{align*}
	\E(M(\mathbf X)\ind_{A^\cpl}) & \leq \frac{k(k-1)}{2n}\sqrt{\frac{1}{n^2}\E\nm{\mathbf X}_4^8}\sqrt{\frac{1}{{n^{2k-4}}}\E\nm{\mathbf X}_2^{4k-8}} \\
    & \leq \frac{k^2 L }{n} \sqrt{\frac{1}{{n^{2k-4}}}\E\nm{\mathbf X}_2^{4k-8}}.
\end{align*} The square root in the second line converges to $1$ (and hence bounded) by \eqref{eq:conv-smaller-norm-mmts}. Thanks to $Lk^2/n \to 0$, it follows that $\E(M(X)\ind_{A^\cpl}) \to 0$, which proves the second part of the lemma.
\end{proof}

%     for any $n$, \[
%         \bigl\Vert\nm{\mathbf X}_{2} -\E\nm{\mathbf X}_2\bigr\Vert_{L^{2k}(\Pr)} \leq C\cdot (2k)
%     \] for some constant $C$. Since $\E\nm{\mathbf X}_2 \leq \sqrt{\E\nm{\mathbf X}_2^2} = \sqrt n$, we can conclude \[
%         \bigl(\E\nm{\mathbf X}_2^{2k}\bigr)^{1/2k} \leq 2Ck + \sqrt n.
%     \] This leads to \[
%         \frac{1}{n^{k}}\E\nm{\mathbf X}_2^{2k} \leq  \biggl(\frac{2Ck}{\sqrt n} + 1\biggr)^{2k}.
%     \] Since $k = o(\sqrt n)$, we have \[
%     \limsup_n \frac{1}{n^k}\E{\nm{\mathbf X}_2^{2k}} \leq 1,
% \]
We remark that $\sqrt{\E \bigl(\frac{1}{n}\sum X_i^4\bigr)^2} = \nm{\frac{1}{n}\sum X_i^4}_{L^2(\Pr)} \leq \frac{1}{n} \cdot n \nm{X_1^4}_{L^2(\Pr)} =  \sqrt{\E X_1^8 }$, so it is sufficient to assume $\E X_1^8 \leq L^2$ when stating this previous lemma. % \fc{may want to mention that the reduction to 2-norm is a tight result}

One might notice that we have dropped the asymptotic correlation condition \[\E(X_1^2X_2^2) - \E(X_1^2)\E(X_2^2) \to 0,\] but are still able to conclude that the convergence of the block moment $\E(X_1^2 \dotsm X_k^2) \to 1$. This is because our new condition implies $\E(X_1^2X_2^2) \to 1$: Take $r = 1$ in \eqref{eq:conv-smaller-norm-mmts} and expand, we would arrive at \[
    \frac{1}{n}\E X_1^4 + \frac{n-1}{n}\E(X_1^2X_2^2) \to 1.
\] Provided that $\E X_1^4 = o(n)$, we have $\E(X_1^2X_2^2) \to 1$.

We are now ready for the proof of \cref{thm:the-norm-cond}.

\begin{proof}[Proof of \cref{thm:the-norm-cond}]
Taking $r = 1$ in \cref{cond:8th-block-mmt} gives $\E X_1^8 \leq L^2$, which implies $\frac{1}{n^2}\E\nm{\mathbf X}_4^8 \leq L^2$ by the comment after the previous lemma. Hence, \cref{thm:main} \cref{cond:expec,cond:var} follow from \cref{cond:2-norm-conv} and the previous lemma. The same argument remains in force if we take $r = 1$ in \cref{cond:4-norm-bdd}.

Let us first assume \cref{cond:8th-block-mmt} and establish \cref{cond:2nd-4th-compare} of \cref{thm:main}.
By \cref{lem:exchangeable-Schur-cvx}, it is straightforward to verify that 
\begin{equation}\begin{split}
    \E(X_1^4\dotsm X_r^4X_{r+1}^2\dotsm X_{2d-r}^2) & \leq \frac{1}{n^{2d-2r}}\E\Bigl(\bigl(X_1^4\dotsm X_r^4\bigr)\bigl( X_1^2 + \dotsb + X_n^2\bigr)^{2d-2r}\Bigr) \\
    & \leq \frac{1}{n^{2d-2r}} \sqrt{\E\bigl(X_1^8 \dotsm X_r^8\bigr)}\sqrt{\E\bigl(X_1^2 +\dotsb +  X_n^2\bigr)^{4d - 4r}}.\end{split} \label{eq:CS-upbd}
    \end{equation} Therefore, to recover \cref{cond:2nd-4th-compare}, it suffices to show for some $L$ that \begin{equation}
    \E(X_1^8\dotsm X_r^8) \frac{1}{n^{4d - 4r}}\E\nm{\mathbf X}_2^{8d-8r} \leq L^{2r}. \label{eq:lhs-Cauchy-Schwarzed}
\end{equation}

Meanwhile, by \cref{lem:2-norm-conv-cond} and \eqref{eq:non-asymp-bound-r-mmt}, our \cref{cond:2-norm-conv} implies that for all large enough $n$ and all corresponding $1\leq r\leq d$, \[
    \frac{1}{n^{4d - 4r}}\E\nm{\mathbf X}_2^{8d-8r}\leq \max\biggl\{\frac{1}{n^{4d}}\E\nm{\mathbf X}_2^{8d},1\biggr\} \leq 1.1.
\]
Thus, if we can show for all large enough $n$, and all $1\leq r\leq d$, that there exists some $L$ such that \[
    \E(X_1^8\dotsm X_r^8) \leq \frac{1}{1.1} L^{2r},
\] our desired estimate in \eqref{eq:lhs-Cauchy-Schwarzed} follow. But this is precisely \cref{cond:8th-block-mmt} when replacing $L$ by $1.1L$, and so \cref{cond:2nd-4th-compare} holds.

To finish the proof, we also need to show \cref{cond:4-norm-bdd} implies \cref{cond:2nd-4th-compare}.
Thus, we again apply \cref{lem:exchange-2nd-mmt-approx} and replace the estimate in \eqref{eq:CS-upbd} by \begin{align*}
    \E(X_1^4\dotsm X_r^4X_{r+1}^2\dotsm X_{2d-r}^2) & \leq \frac{1}{n^{r}}\frac{1}{n^{2d-2r}}\E\Bigl(\bigl(X_1^4 + \dotsb +X_n^4\bigr)^r\bigl( X_1^2 + \dotsb + X_n^2\bigr)^{2d-2r}\Bigr) \\
    & \leq \frac{1}{n^{r}}\frac{1}{n^{2d-2r}} \sqrt{\E\bigl(X_1^4 + \dotsb +X_n^4\bigr)^{2r}}\sqrt{\E\bigl(X_1^2 +\dotsb +  X_n^2\bigr)^{4d - 4r}}.
    \end{align*} Repeating the same argument above leads to \[
        \frac{1}{n^{2r}} \E\nm{\mathbf X}_4^{8r} \leq \frac{1}{1.1} L^{2r},
    \] and hence \cref{cond:4-norm-bdd} is also sufficient for applying \cref{thm:main}. % \dm{was condition (ii) ever used?}
\end{proof}

The conditions in \Cref{thm:the-norm-cond}\ref{cond:2-norm-conv}\ref{cond:4-norm-bdd} can be achieved by imposing concentration properties on the norm, which is the content of \cref{prop:subexponential-Gaussian-norm}.
% The proof for $d = o(\sqrt n)$ and uniformly bounded log-Sobolev constant is precisely the same. If $\mathbf X$ satisfies the log-Sobolev inequality, then we would have sub-Gaussian concentration for Lipschitz functions of $\mathbf X$.

% \begin{prop}
%     Assume $d = o(\sqrt n)$. Suppose \begin{enumerate}[label=(\alph*')]
%         \item $\E(X_1^2) = 1$,
%         \item $\E X_1^4$ is uniformly bounded,
%         \item $\nm{\mathbf X}_2$ is subexponential with proxy variance uniformly bounded for all $n$, and 
%         \item $\nm{\mathbf X}_4^2$ is subexponential with proxy variance uniformly bounded for all $n$.
%     \end{enumerate} Then \cref{cond:expec,cond:var,cond:2nd-4th-compare} follows. \dmL{can we relax these conditions if now we $L$ to grow with $n$? say $L = n^{\varepsilon}$. FC: Most likely the fourth moment may be allowed to grow with $n$ instead of being uniformly bounded} 
% \end{prop} 
\begin{proof}[Proof of \cref{prop:subexponential-Gaussian-norm}]
To take advantage of the relationship between $d$ and $n$, we use the following trick: 
    \begin{equation} \begin{split}
        \E\biggl(\frac{\nm{\mathbf X}_2^{8d}}{n^{4d}}\biggr) &= \E\exp\biggl(8d\log\Bigl(\frac{\nm{\mathbf X}_2}{\sqrt{n}}\Bigr)\biggr) \leq \E\exp\biggl(8d\Bigl(\frac{\nm{\mathbf X}_2}{\sqrt{n}}-1\Bigr)\biggr) \\ 
 	& =\E \exp\biggl(\frac{8d}{\sqrt{n}}(\nm{\mathbf X}_2-\sqrt{n})\biggr) \\
    & =\exp\biggl(\frac{8d}{\sqrt{n}}(\E\nm{\mathbf X}_2 - \sqrt n)\biggr) \E\exp\biggl(\frac{8d}{\sqrt{n}}(\nm{\mathbf X}_2-\E\nm{\mathbf X}_2)\biggr) \\
    & \leq \E\exp\biggl(\frac{8d}{\sqrt{n}}(\nm{\mathbf X}_2-\E\nm{\mathbf X}_2)\biggr), \label{eq:upper-bound-normalized-norm}\end{split}
\end{equation}
where the last inequality comes from the observation  
\[
    \E\nm{\mathbf X}_2 \leq \bigl(\E\nm{\mathbf X}_2^2\bigr)^{1/2} =  \bigl(n\E X_1^2\bigr)^{1/2} = \sqrt n.
\]

Now assume $\nm{\mathbf X}_2$ is subexponential with uniformly bounded proxy variance $2\sigma^2$ for $X$, i.e., \[
    \E\exp\bigl(\lambda (\nm{\mathbf X}_2 - \E \nm{ \mathbf X}_2)\bigr) \leq \exp\bigl(\lambda^2 \sigma^2\bigr)
\] provided that $ \lambda^2 \leq 1/\sigma^2$.
Since $d^2/n = o(1)$, we can conclude from \eqref{eq:upper-bound-normalized-norm} that for all sufficiently large $n$, \[
    1\leq \frac{1}{n^{4d}}\E{\nm{\mathbf X}_2^{8d}} \leq \exp\biggl(\frac{64d^2}{n}\sigma^2\biggr),
\] which proves $\frac{1}{n^{4d}}\E \nm{\mathbf X}_2^{8d} \to 1$.

For $\frac{1}{n^{2r}} \E\nm{\mathbf X}_4^{8r} \leq L^{2r}$, we use a slightly different approach. By assumption, let $\sup_n \E X_1^4 \leq K$ for some constant $K > 0$. Saying the $\ell^4$-norm is subexponential with proxy variance $2\sigma^2$ is equivalent to saying \begin{equation}
    \bigl\Vert\nm{\mathbf X}_{4} -\E\nm{\mathbf X}_4\bigr\Vert_{L^{8r}(\Pr)} \leq C\cdot (8r) \label{eq:subexponential-Lp-cond}
\end{equation} for some constant $C$ dependent on the proxy variance $2\sigma^2$; see \cite[Proposition~2.8.1]{Vershynin_2026}. Also \[
    \E\nm{\mathbf X}_4 = \E\bigl(\bigl(X_1^4 + \dotsb + X_n^4\bigr)^{1/4}\bigr)
    \leq \bigl(n \E X_1^4\bigr)^{1/4} \leq \bigl(nK\bigr)^{1/4}.
\] Hence \[
    \bigl(\E\nm{\mathbf X}_4^{8r}\bigr)^{1/8r} = \bigl\Vert\nm{\mathbf X}_{4}\bigr\Vert_{L^{8r}(\Pr)} \leq 8Cr + (nK)^{1/4},
\] which leads to \[
    \frac{1}{n^{2r}}\E\nm{\mathbf X}_4^{8r} \leq \frac{1}{n^{2r}}\bigl(8Cr + (nK)^{1/4}\bigr)^{8r} \leq \biggl(\frac{8Cr}{n^{1/4}} + K^{1/4}\biggr)^{8r} \leq \max\biggl\{\Bigl(\frac{8Cd}{n^{1/4}} + K^{1/4}\Bigr)^4,1\biggr\}^{2r} % \leq 2^{4r}\bigl[(4Cr)^{4r} + (nL)^{2r}\bigr] = \bigl(64C^2r^2\bigr)^{2r} + (4nL)^{2r}. 
\]
If $d= o(n^{1/3})$, then \[
    L \coloneqq \max\biggl\{\Bigl(\frac{8Cd}{n^{1/4}} + K^{1/4}\Bigr)^4,1\biggr\} = o(n^{1/3})
\] as well. Since $Ld^2 / n \to 0$, the proof is complete.\footnote{The reader may have noticed that it is possible to allow the subexponential proxy variance and the fourth moment bound to grow with $n$, while still maintaining the growth condition $d = o(n^{1/3})$ that we obtained. For simplicity, we have ignored this in our proposition statement.}

If instead $\nm{\mathbf X}_4$ is sub-Gaussian with proxy variance $2\sigma^2$, then we have the stronger \[
    \bigl\Vert\nm{\mathbf X}_{4} -\E\nm{\mathbf X}_4\bigr\Vert_{L^{8r}(\Pr)} \leq C\cdot \sqrt{8r}
\] on top of \eqref{eq:subexponential-Lp-cond}; see \cite[Proposition~2.6.1]{Vershynin_2026}. Following through the same calculation leads us to \[
    \frac{1}{n^{2r}}\E\nm{\mathbf X}_4^{8r} \leq \max\biggl\{\Bigl(\frac{C\sqrt{8d}}{n^{1/4}} + K^{1/4}\Bigr)^4,1\biggr\}^{2r}.
\] Taking $d = o(\sqrt n)$, then \[
    L \coloneqq \max\biggl\{\Bigl(\frac{C\sqrt{8d}}{n^{1/4}} + K^{1/4}\Bigr)^4,1\biggr\} = O(1).
\] Again $Ld^2 / n \to 0$, and is certainly optimal since $d$ can be at most $o(\sqrt n)$.
\end{proof}

\section{Theorem~\ref{thm:tensor-power} for the the tensor power model} \label{sec:tensor-power}

\begin{proof}[Proof sketch of \cref{thm:tensor-power}]
    From our discussion in \cref{sec:ext-tensor-power}, it suffices to show that \begin{equation}\label{eq:reducedT-esd-conv}\ESD(K_{\ReducedT}) \wkconv \mu_{\MP(c)}\end{equation} with probability $1$. Let $p=p_n = \binom{n}{d}$. In the proof of \cref{thm:main} in \cref{sec:concentration-proof-of-main}, we have verified that for any sequence of PSD matrices $A_p \in \R^{p\times p}$ with $\opnm{A_p} = O(1)$, it holds for $\mathbf x\sim \PrincipalT(n,d,\mathbf X)$ that \begin{equation}
        \frac{\mathbf x^\trp A_p \mathbf x - \tr A_p}{p} \to 0 \quad\text{in probability}. \label{eq:prinipalT-conv-in-prob}
    \end{equation} Notice that given the same base vector $\mathbf X$, the reduced symmetric tensor $\mathbf y \sim \ReducedT(n,d,\mathbf X)$ contains the principal tensor $\mathbf x \sim \PrincipalT(n,d,\mathbf X)$ as a sub-vector. Also recall from \cref{lem:sqrt-n-lim} that if $d = o(\sqrt n)$, then $\binom{n+d-1}{d} \big/ \binom{n}{d} \to 1$. Therefore, one can expect an easy extension from \eqref{eq:prinipalT-conv-in-prob} to \begin{equation} \label{eq:reducedT-conv-in-prob}
        \frac{\mathbf y^\trp B_{\tilde p} \mathbf y - \tr B_{\tilde p}}{\tilde p} \to 0 \quad\text{in probability},
    \end{equation} where $\tilde p = \binom{n+d-1}{d}$, and $B_{\tilde p} \in \R^{\tilde p \times \tilde p}$ is a sequence of PSD matrices with $\opnm{B_{\tilde p}} = O(1)$. After establishing \eqref{eq:reducedT-conv-in-prob}, we can conclude \eqref{eq:reducedT-esd-conv} by \cref{thm:MP}.

    The full argument for \eqref{eq:reducedT-conv-in-prob} is almost exactly the same as \cite[Proposition~3.3]{Yaskov_2023}, where the base vector $\mathbf X$ was assumed to have independent components. However, we do need to verify that our exchangeable and unconditional base vector $\mathbf X$ indeed satisfy
\[
    \biggl(\frac{1}{n} \sum_{\alpha = 1}^n X_\alpha^2\biggr)^d = \frac{1}{n^{d}} \nm{\mathbf X}_2^{2d} \to 1 \quad \text{in probability.}
\]
By Chebyshev's inequality, it is sufficient to verify the $L^2$ convergence \[
    \E \biggl(\frac{\nm{\mathbf X}_2^{2d}}{n^{d}} - 1 \biggr)^2 = \frac{1}{n^{2d}}\E\nm{\mathbf X}_2^{4d} - 2 \frac{1}{n^d}\E\nm{\mathbf X}_2^{2d} + 1 \to 0.
\] This is true under the assumptions of either \cref{thm:finite-case} or \cref{thm:the-norm-cond}, where we use \cref{lem:suff-cond-prod-approx} or \cref{lem:2-norm-conv-cond}, respectively.
% The following fact is also needed. In the proof of \cref{thm:main}, if we take $A_p$ to be the identity in \eqref{eq:conv-in-prob} for any exchangeable and unconditional base vector $\mathbf X$, it holds that $\frac{1}{\binom{n}{d}} \mathbf x^\trp \mathbf x \to 1$ in probability. 
% Second, for any exchangeable and unconditional base vector $\mathbf X$, we need to show that \[\frac{\mathbf x^\trp\mathbf x}{\binom{n}{d}} = \frac{1}{\binom{n}{d}}\sum_{j \in \binom{[n]}{d}} \prod_{\alpha \in j} X_\alpha^2 \to 1 \quad\text{in probability.}\] (Previously this holds for i.i.d.\ base vectors, which we discussed in \cref{sec:two-tensor-models}.) In \cref{thm:main}, we assumed $\E(X_1^2\dotsm X_{d}^2) \to 1$, and therefore $\frac{1}{\binom{n}{d}}\E \mathbf x^\trp\mathbf x \to 1$. Also recall we showed in the proof of \cref{thm:main} that $\Var(\mathbf x^\trp \mathbf x) = o\bigl(\binom{n}{d}^2\bigr)$, and therefore $\frac{1}{\binom{n}{d}}\mathbf x^\trp\mathbf x \to 1$ in $L^2$.
\end{proof}

% We now give an outline that shows the ESD for the reduced symmetric tensor model that we defined in  the ESD for the principal tensor model are the same. First notice that when $d = o(\sqrt n)$, $\binom{n+d-1}{d} \big/ \binom{n}{d} \to 1$. Therefore we should not expect the reduced symmetric tensor model to be too different from the principal tensor model. \fc{reduced symmetric tensor model}

% To be precise, let $A_p$ be a PSD matrix with $\nm{A_p} \leq 1$. If we let $\mathbf x_p$ be the reduced symmetric tensor in $\R^{\binom{n+d-1}{d}}$, we can break up \[
%     \frac{\mathbf x_p^\trp A_p \mathbf x_p - \tr A_p}{p}
% \] into three pieces by breaking up $\mathbf x_p$ and $A_p$, and verify that all three pieces indeed converge to $0$ in probability. \dm{This is way way way too vague. Either be concrete in what we say or just state the consequence that we use.}

\section{Verification of examples} \label{sec:examples}
% \subsection{Signed permutations}

\subsection{Examples satisfying Theorem~\ref{thm:the-norm-cond}}

The results for measures satisfying the Poincar\'e and log-Sobolev inequality quickly follows from \cref{prop:subexponential-Gaussian-norm} via Lipschitz concentration.

\begin{proof}[Proof of \cref{prop:Poincare,prop:log-Sobolev}]
    Assume $\mathbf X$ has a Poincar\'e constant uniformly bounded by $C > 0$. Since the $\ell^2$ and $\ell^4$ norms are $1$-Lipschitz, by Lipschitz concentration \cite[Corollary~3.2, Proposition~1.8]{Ledoux_2001}, we have for both norms \[
        \Pr\bigl(\bigl\vert \nm{\mathbf X} - \E\nm{\mathbf X}\bigr\vert \geq t\bigr) \leq c_1\exp\biggl(\frac{-c_2t}{\sqrt C}\biggr)
    \] for any $t \geq 0$ and some absolute constants $c_1,c_2 > 0$. This is equivalent to saying that $\nm{\mathbf X}_2$ and $\nm{\mathbf X}_4$ are both subexponential with uniformly bounded proxy variance, which proves our proposition by \cref{prop:subexponential-Gaussian-norm}.

    Also note that projection to the coordinate is $1$-Lipschitz as well, and therefore by Lipschitz concentration, \[
        \Pr(\lvert X_1 - \E X_1\rvert \geq t) \leq c_1\exp\biggl(\frac{-c_2t}{\sqrt C}\biggr).
    \] Since $\E X_1 = 0$, by \cite[Proposition~2.8.1]{Vershynin_2026} this implies $\nm{X_1}_{L^4(\Pr)}\leq C_1$ for some $C_1$ dependent on $C$. Therefore, we also have uniform boundedness of the fourth moment $\sup_n \E X_1^4 < \infty$, as required in \cref{prop:subexponential-Gaussian-norm}.

    The proof for $d = o(\sqrt n)$ and $\mathbf X$ having the stronger uniformly bounded log-Sobolev constant is precisely the same. If $\mathbf X$ satisfies the log-Sobolev inequality, then we would have sub-Gaussian concentration for Lipschitz functions of $\mathbf X$; see \cite[Theorem~5.3]{Ledoux_2001}. 
\end{proof}

% Therefore \[
%     \bigl\Vert\nm{\mathbf X}_{4}^2 -\E\nm{\mathbf X}_4^2\bigr\Vert_{L^{4r}(\Pr)} \leq C\cdot (4r)
% \] for some constant $C$. Also \[
%     \E\nm{\mathbf X}_4^2 = \E\sqrt{X_1^4 + \dotsb + X_n^2} 
%     \leq \sqrt{n \E X_1^4} \leq \sqrt{nL}.
% \] Hence \[
%     \bigl(\E\nm{\mathbf X}_4^{8r}\bigr)^{1/4r} = \bigl\Vert\nm{\mathbf X}_{4}^2\bigr\Vert_{L^{4r}(\Pr)} \leq 4Cr + \sqrt{nL},
% \] which leads to \[
%     \frac{1}{n^{2r}}\E\nm{\mathbf X}_4^{8r} \leq \frac{1}{n^{2r}}\bigl(4Cr + \sqrt{nL}\bigr)^{4r} \leq \biggl(\frac{4Cr}{\sqrt n} + \sqrt{L}\biggr)^{4r}% \leq 2^{4r}\bigl[(4Cr)^{4r} + (nL)^{2r}\bigr] = \bigl(64C^2r^2\bigr)^{2r} + (4nL)^{2r}. 
% \]

% Dividing the two sides by $n^{2r}$ then gives us \[
%     \frac{1}{n^{2r}}\E\nm{\mathbf X}_4^{8r} \leq \biggl(\frac{Kr^2}{n}\biggr)^{2r} + (4L)^{2r}.
% \] 
% Since $d^2/n \to 0$, it must be the case that $\frac{4Cr}{\sqrt n} \leq \frac{4Cd}{\sqrt n} \leq 1$ for large enough $n$ and all $0\leq r\leq d$. Therefore we have proven \eqref{eq:new-target} after replacing $L$ by $(\sqrt L+1)^2$.

% & \leq 2^{4r} \bigl[\E\nm{\mathbf X}_4^{8r} + \bigl(\E\nm{\mathbf X}_4^{2}\bigr)^{4r}\bigl] \\& \leq 

We now prove \cref{prop:log-concave} for the case where $\mathbf X$ is isotropic log-concave. First, we show that such $\mathbf X$ satisfy \cref{thm:the-norm-cond} \cref{cond:2-norm-conv}.

% \begin{proof}[Proof of \cref{prop:log-concave}]
%     We first show for any $d = o(\sqrt{n /\log n})$, $\frac{1}{n^{4d}}\E\nm{\mathbf X}_2^{8d} \to 1$.
% \end{proof}

% First of all, the uniform bound on the fourth moment of the marginal is a direct consequence of Borell's lemma \dm{Need to say what Borell's lemma is and supply a reference}, which is the content of \cref{cond:uniform-4th-mmt}.

% \dm{This is a bit messy. It will be better to have a lemma showing this, and in the proof use Fleury. No need to state Fleury as a separate lemma.}We will now show that for any $d = o(\sqrt{n /\log n})$, $\frac{1}{n^{4d}}\E\nm{\mathbf X}_2^{8d} \to 1$. \fc{part 1 of a large proof environment}

\begin{lem}
    For any $d = o(\sqrt{n /\log n})$, it holds that $\frac{1}{n^{4d}}\E\nm{\mathbf X}_2^{8d} \to 1$.
\end{lem}
\begin{proof}
Our starting point is the following result from \cite[Lemma~4]{Fleury_2010_AIHP}. Let $C>0$ be the Poincar\'e constant of the isotropic log-concave random vector $\mathbf{X}$. %Meanwhile, $\nm{\mathbf X}_2$ has finite moments of all orders.
Then, there exist absolute constants $c_1,c_2>0$ such that \[
            \bigl(\E \nm{\mathbf X}_2^k\bigr)^{1/k} \leq \biggl(1 + \frac{c_2 k C}{n}\biggr) \sqrt{n}\quad \text{for all }2\leq k \leq c_1\frac{\sqrt n}{\sqrt C}.
        \] .

% \begin{proof}[Proof of \cref{exa:log-concave}]
    By \cite{klartag2023logarihtmic}, for some absolute constant $c_3 > 0$, the Poincar\'e constant $C \leq c_3\log n$ for any isotropic log-concave measure in $\R^n$. Replacing $k$ by $8d$ in the previous proposition, 
    \[
        \E\biggl(\frac{\nm{\mathbf X}_2}{\sqrt n}\biggr)^{8d} \leq \biggl(1 + \frac{c_4 \cdot 8d}{n/\log n}\biggr)^{8d}
    \] for all $1\leq d \leq c_5 \sqrt{n / \log n}$. In particular, for any $d = o(\sqrt{n / \log n})$, \[
        \limsup_n \E\biggl(\frac{\nm{\mathbf X}_2}{\sqrt n}\biggr)^{8d} \leq 1.
    \] This proves $\frac{1}{n^{4d}}\E\nm{\mathbf X}_2^{8d} \to 1$, since by Jensen's inequality \[
        \E\biggl(\frac{\nm{\mathbf X}_2}{\sqrt n}\biggr)^{8d} \geq \frac{1}{n^{4d}}\bigl(\E\nm{\mathbf X}^2_2\bigr)^{4d} = 1.\qedhere
    \]
\end{proof}   
    Note that assuming the correctness of the KLS conjecture \cite{KLS_1995}, this above argument would give us $\frac{1}{n^{4d}}\E\nm{\mathbf X}_2^{8d} \to 1$ for $d = o(\sqrt n)$, and should then be optimal.

It remains to verify either \cref{cond:8th-block-mmt} or \cref{cond:4-norm-bdd} for all $Ld^2 / n \to 0$ and $1\leq r\leq d$. The following proposition shows that \cref{cond:8th-block-mmt} is correct for some absolute constant $L > 0$ and all $d = o(\sqrt n)$. We also believe the proposition itself might be of independent interest.

\begin{prop} \label{prop:isotropic-lc-8th-block-mmt}
    For any unconditional isotropic log-concave vector $\mathbf X$, there exists an absolute constant $L$ such that \[
        \E(X_1^8\dotsm X_r^8) \leq L^{2r}
    \] for all $1 \leq r \leq n$. (In general, we have for any fixed exponent $1 \leq q < \infty$, $\nm{X_1\dotsm X_r}_{L^q(\Pr)} \leq L^r$.)
\end{prop}
\begin{proof}
    We adapt a slick argument that appeared in \cite[Proposition~3.1]{Bobkov_Nazarov_2003}. Let $\mathbf X = (X_1,\dotsc,X_n)$ follow the log-concave unconditional density $f$ in $\R^n$. Define $g$ to be the marginal of $f$ in an $r$-dimensional subspace. Now let $g^+(y) = g(\frac{1}{2}y)$ for $y \in [0,\infty)^r$, which is the new density for the random vector $(2\abs{X_1},\dotsc,2\abs{X_r})$. This $g^+$ continues to be a log-concave density, and is decreasing in each coordinate.

    It follows that \[
    g^+(y)\prod_{j=1}^r y_j = g^+(y) \int_0^{y_1}\dotsi \int_0^{y_n} 1\,dx \leq \int_0^{y_1}\dotsi \int_0^{y_n} g^+(x)\,dx \leq 1,
\] and therefore $\log \prod_{j=1}^r y_j \leq -\log g^+(y)$ for all $y \in [0,\infty)^r$. (Here we define $\log 0 = -\infty$.) This is a concave function bounded above by a convex function. Write them respectively as $v(y)$ and $u(y)$, and there exists a linear function such that \[
    v(y) \leq h(y) \leq u(y) \quad\text{for all }y\in [0,\infty)^r,
\] while satisfying \[
    h(y) = \log \prod_{j=1}^r a_j + \biggl\langle{\grad \log \prod_{j=1}^r a_j},y - a\biggr\rangle,
\] i.e., $h$ is tangent to the concave function $v$ at the point $a = (a_1,\dotsc,a_r) \in (0,\infty)^r$.

Now take $\lambda_j = 1 / a_j$, the inequality $u(y) \geq h(y)$ simplifies to \[
    g^+(y) \leq e^r \prod_{j=1}^r \lambda_j e^{-\lambda_j y_j}.
\] 
% \dm{This looks a bit weird. Either something is true or it's not. If it's true it doesn't matter who mentions it.}
Since $\int x\, g(x)\,dx = 0$ and $\varphi \coloneqq -\log g$ is convex, \[
    -\log g(0) = \varphi(0) = \varphi\biggl(\int x\,g(x)\,dx \biggr) \leq \int \varphi(x) g(x)\,dx = -\int g\log g.
\] Recall over all isotropic densities, the differential entropy $g\mapsto -\int g\log g$ is maximized when $g$ is the standard Gaussian density. Hence \[g^+(0) = g(0)\geq \exp\biggl(\int g \log g\biggr) \geq \biggl(\frac{1}{\sqrt{2\pi e}}\biggr)^r\] (this fact was discussed in the introduction of \cite{klartag2025affirmative}). It follows that $\prod_{j=1}^r \lambda_j \geq \bigl(\frac{1}{e\sqrt{2\pi e}}\bigr)^r$. Therefore \begin{align*}
     2^{8r} \E(X_1^8\dotsm X_r^8) = \int \prod_{j=1}^r y_j^8\, g^+(y)\,dy & \leq \int \prod_{j=1}^r y_j^8 \biggl(e^r \prod_{j=1}^r \lambda_j e^{-\lambda_j y_j}\biggr)\,dy \\ & = e^r \prod_{j=1}^r \frac{8!}{\lambda_j^8} \leq (8!\,e)^r \bigl(e\sqrt{2\pi e}\bigr)^{8r}.
\end{align*} This means we have found an absolute constant $L$ such that $\E(X_1^8\dotsm X_r^8) \leq L^{2r}$.
\end{proof}
The proof of \cref{prop:log-concave} is now immediate.
\begin{proof}[Proof of \cref{prop:log-concave}]
    We have just verified \cref{thm:the-norm-cond} conditions~\ref{cond:2-norm-conv}\ref{cond:8th-block-mmt}, and therefore we conclude that when $d = o(\sqrt{n /\log n})$ and the base vector $\mathbf X$ is an isotropic, exchangeable, unconditional, and log-concave vector, $\ESD(K_{\PrincipalT})\wkconv \mu_{\MP(c)}$ with probability $1$.
\end{proof}

\begin{rem*}
We briefly discuss that using \cref{cond:4-norm-bdd} would instead give us the weaker $d = o(n^{1/3})$, and we cannot improve it any further. In fact, the reader will see that in contrast to \cref{prop:isotropic-lc-8th-block-mmt}, which holds universally for any $r \leq n$, the bound $\frac{1}{n^{2r}}\E \nm{\mathbf X}_4^{8r} \leq L^{2r}$ is true for some \emph{constant} $L$ independent of $n$ if and only if $r \leq d = o(n^{1/4})$.

     For general $\ell^k$ norms, according to \cite[Theorem~2]{Latala_modified_Paouris}, $\bigl(\E \nm{\mathbf X}_4^{8r} \bigr)^{1/8r}\leq C(4n^{1/4} + 8r)$. This leads to \[
        \frac{1}{n^{2r}}\E \nm{\mathbf X}_4^{8r} \leq \biggl[C^4 \Bigl(4 + \frac{8r}{n^{1/4}}\Bigr)^{4}\biggr]^{2r} \leq \max\biggl\{C^4 \Bigl(4 + \frac{8d}{n^{1/4}}\Bigr)^{4},1\biggr\}^{2r}.
    \] Set this maximum to be our $L$. When $d = o(n^{1/3})$, the maximum is also $o(n^{1/3})$, and so $Ld^2 / n\to 0$.
    
    If we use \cref{cond:4-norm-bdd}, we cannot improve the growth condition beyond $d = o(n^{1/3})$ for log-concave measures and for measures satisfying the Poincar\'e inequality.
    
    Consider the product of isotropic Laplace distributions, with density $\frac{1}{\sqrt 2}e^{-\sqrt 2\abs{x}}\,dx$. This is a log-concave measure with Poincar\'e constant $2$, but does not satisfy the log-Sobolev inequality. In our case, \[
    \bigl(\E\nm{\mathbf X}_4^{8r}\bigr)^{1/8r} \geq \bigl(\E X_1^{8r}\bigr)^{1/8r} \geq \frac{1}{C_1}8r
\] for some absolute positive constant $C_1$. The last inequality comes from $\E X_1^{8r} = (8r)! / (\sqrt{2})^{8r}$, and Stirling's approximation. It follows that \[
    \frac{1}{n^{2r}}\E\nm{\mathbf X}_4^{8r} \geq \biggl[\Bigr(\frac{8r}{C_1n^{1/4}}\Bigr)^{4}\biggr]^{2r}.
\] If $r \leq d$ is on the order of $n^{1/3}$, then the term inside the bracket would also be on the order of $n^{1/3}$. This would violate our requirement in \cref{thm:the-norm-cond}\ref{cond:4-norm-bdd} that $\frac{1}{n^{2r}}\E\nm{\mathbf X}_4^{8r} \leq L^{2r}$ for some sequence $L$ satisfying $Ld^2 = o(n)$. However, by \cref{thm:indep-tensor}, for base vectors following the product of Laplace distributions, $\ESD(K_{\PrincipalT})$ converges to the MP law for any $d = o(\sqrt n)$.
\end{rem*}

\subsection{Explicit computation for \texorpdfstring{$\ell^k$}{lk} spherical distributions}

For the $\Inv_k(R)$ distributions, which include the uniform measure on the $\ell^k$ spheres and balls, we start by computing the block moments.

\begin{lem} \label{lem:moments-lp-sphere}
    For $0<k<\infty$, let $(X_1,\dotsc,X_n) \sim \Inv_k(R)$. Then \[
        \E(X_1^4\dotsm X_r^4 X_{r+1}^2\dotsm X_{2d - r}^2) = \frac{\Gamma(5/k)^r \,\Gamma(3/k)^{2d - 2r}}{\Gamma(1/k)^{2d - r}} \cdot \frac{\Gamma\bigl(\frac{n}{k}\bigr)}{\Gamma\bigl(\frac{n+4d}{k}\bigr)}\cdot \E(R^{4d}).
    \]
    Let $r = 0$ and replace $2d$ by $d$, we get \[
        \E(X_1^2\dotsm X_d^2) = \frac{\Gamma(3/k)^d}{\Gamma(1/k)^d} \cdot \frac{\Gamma\bigl(\frac{n}{k}\bigr)}{\Gamma\bigl(\frac{n+2d}{k}\bigr)}\cdot \E(R^{2d}),
    \] and taking $d = 1$ gives \[
        \E(X_1^2) = \frac{\Gamma(3/k)}{\Gamma(1/k)} \cdot \frac{\Gamma\bigl(\frac{n}{k}\bigr)}{\Gamma\bigl(\frac{n+2}{k}\bigr)}\cdot \E(R^{2}).
    \]
\end{lem}
\begin{proof}
    Let $(Z_1,\dotsc,Z_n)$ be an $n$-dimensional $k$-Gaussian independent of $R$. \[\E(X_1^4\dotsm X_r^4 X_{r+1}^2\dotsm X_{r+s}^2) = \E(R^{4d})\E\biggl(\frac{Z_1^4\dotsm Z_r^4 Z_{r+1}^2\dotsm Z_{r+s}^2}{\nm{\mathbf Z}_k^{4d}}\biggr).\]
    By independence between $\nm{\mathbf Z}_k$ and $\bigl(\frac{Z_1}{\nm{\mathbf Z}_k},\dotsc,\frac{Z_n}{\nm{\mathbf Z}_k}\bigr)$, we obtain \[
        \E\bigl(\nm{\mathbf Z}^{4d}_k\bigr) \E\biggl(\frac{Z_1^4\dotsm Z_r^4 Z_{r+1}^2\dotsm Z_{2d-r}^2}{\nm{\mathbf Z}^{4d}_k}\biggr) = \E(Z_1^4\dotsm Z_r^4 Z_{r+1}^2\dotsm Z_{2d-r}^2) = \bigl(\E Z_1^4\bigr)^r\cdot \bigl(\E Z_1^2\bigr)^{2d - 2r},\] which is equal to \[
            k^{4d / k} \frac{\Gamma(5/k)^r \,\Gamma(3/k)^{2d - 2r}}{\Gamma(1/k)^{2d - r}}.
        \]
    Note that $\nm{\mathbf Z}_k^k \sim \Gam(n / k,k)$, where $n/k$ is the shape parameter and $k$ is the scale parameter. Its $(4d/k)$-th moment is $k^{4d / k} \frac{\Gamma(\frac{n+4d}{k})}{\Gamma(\frac{n}{k})}$. Putting all the equations above together gives us the desired formula in the lemma.
\end{proof}

To prove \cref{prop:lk-general-cond}, we need to know the asymptotic ratio of Gamma functions that appeared in the previous lemma.
\begin{lem}
    For real numbers $x > 0$ and $\alpha \geq 1$, with $\alpha = o(\sqrt x)$, we have \[
        \frac{\Gamma(x + \alpha)}{\Gamma(x)} = x^\alpha \biggl[1 + \frac{\alpha^2 - \alpha}{2x} + O\biggl(\frac{\alpha^4}{x^2}\biggr)\biggr] = x^\alpha \bigl(1+o(1)\bigr).
    \]
\end{lem}
Traditionally this result was only stated for fixed $\alpha$, but was proved in general for complex-valued $x$ and $\alpha$ (provided the Gamma functions are well-defined); see for example \cite{Tricomi_Erdelyi_1951}. Here we consider $\alpha$ that may grow with $x$. 

\begin{proof}
For the log Gamma function, we have the following asymptotic Stirling's formula 
    \[
    \log \Gamma(x) = \biggl(x  - \frac{1}{2}\biggr) \log x - x + \log \sqrt{2\pi} + \frac{1}{12x} + O(x^{-3});
    \] see \cite[page~368]{Gamelin_complex}. Noticing that $\log (x+\alpha) = \log x + \log(1 + \frac{\alpha }{x})$, we then have \begin{equation} \label{eq:log-diff-Gamma}
    \log \Gamma(x+\alpha) - \log \Gamma(x) = \alpha \log x + \biggl(x + \alpha - \frac{1}{2}\biggr) \log\biggl(1 + \frac{\alpha}{x}\biggr) - \alpha +O(\alpha x^{-2}) + O(x^{-3}).
\end{equation}
Let us focus on the middle term \[
    \biggl(x + \alpha - \frac{1}{2}\biggr) \log\biggl(1 + \frac{\alpha}{x}\biggr) - \alpha.
\] First, by $\alpha = o(\sqrt x)$, Taylor's theorem gives us \[
    \log\biggl(1 + \frac{\alpha}{x}\biggr)  = \frac{\alpha}{x} - \frac{\alpha^2}{2x^2} + O\biggl(\frac{\alpha^3}{x^3}\biggr).
\] One can then verify that given $\alpha  = o(\sqrt x)$, \[
    \biggl(x + \alpha - \frac{1}{2}\biggr) \biggl[\frac{\alpha}{x} - \frac{\alpha^2}{2x^2} + O\biggl(\frac{\alpha^3}{x^3}\biggr)\biggr] - \alpha = \frac{\alpha^2 - \alpha}{2x} + O\biggl(\frac{\alpha^3}{x^2}\biggr)
\] Returning back to \eqref{eq:log-diff-Gamma}, we now have \begin{equation}
    \log\frac{\Gamma(x + \alpha)}{\Gamma(x)} = \alpha \log x + \frac{\alpha^2 - \alpha}{2x} + O\biggl(\frac{\alpha^3}{x^2}\biggr).
    \end{equation}
Thus, \begin{align*}
    \frac{\Gamma(x + \alpha)}{\Gamma(x)} & = x^\alpha \exp\biggl( \frac{\alpha^2 - \alpha}{2x} + O\biggl(\frac{\alpha^3}{x^2}\biggr)\biggr) \\
    & = x^\alpha\biggl[1 + \frac{\alpha^2 - \alpha}{2x} + O\biggl(\frac{\alpha^3}{x^2}\biggr) + O\biggl(\Bigl(\frac{\alpha^2 - \alpha}{2x}\Bigr)^2 + O\Bigl(\frac{\alpha^6}{x^4}\Bigr)\biggr)\biggr] \\
    & = x^\alpha \biggl[1 + \frac{\alpha^2 - \alpha}{2x} + O\biggl(\frac{\alpha^4}{x^2}\biggr)\biggr],
\end{align*} where we used Taylor's theorem in the second line.
\end{proof}

\begin{proof}[Proof of \cref{prop:lk-general-cond}]

    % First, it was proven in \cite{Tricomi_Erdelyi_1951} that for any $x,\alpha\in \R$, \[\frac{\Gamma(x + \alpha)}{\Gamma(x)} = x^{\alpha - \beta} \biggl(1 + \frac{1}{2x}(\alpha)(\alpha - 1) + O(x^{-2})\biggr),\] provided that the two Gamma functions are defined. \fc{fix this, take $n/k$ out} 
    
%     It follows that \begin{align*}
%     \frac{\Gamma\bigl(\frac{n}{k}\bigr)}{\Gamma\bigl(\frac{n+2d}{k}\bigr)} & = \biggl(\frac{n}{k}\biggr)^{-2d/k} \biggl(1 + \frac{k}{2n} \Bigl(-\frac{2d}{k}\Bigr)\Bigl(\frac{2d}{k} - 1\Bigr) + O\Bigl(\frac{k^2}{n^2}\Bigr)\biggr) \\
%     & = \biggl(\frac{n}{k}\biggr)^{-2d/k} \biggl(1 - \frac{2}{k}\cdot \frac{d^2}{n} + \frac{d}{n} + O\Bigl(\frac{k^2}{n^2}\Bigr)\biggr).
% \end{align*}

For fixed $k<\infty$, given $d = o(\sqrt n)$, by the previous proposition, we then have \[
    \frac{\Gamma\bigl(\frac{n}{k}\bigr)}{\Gamma\bigl(\frac{n+2d}{k}\bigr)} =\biggl(\frac{n}{k}\biggr)^{-2d/k}\bigl(1+o(1)\bigr) \quad \text{and}\quad
    \frac{\Gamma\bigl(\frac n k\bigr)}{\Gamma \bigl(\frac{n+4d}{k}\bigr)} =\biggl(\frac{n}{k}\biggr)^{-4d/k}\bigl(1+o(1)\bigr).
\]
\Cref{lem:moments-lp-sphere}, along with the convergence of ratios between Gamma functions, allows us to conclude that \cref{cond:expec-lk,cond:var-lk} in our proposition imply \cref{thm:main} \cref{cond:expec,cond:var}, respectively.% \footnote{in fact, probably equivalent}

Lastly, \cref{thm:main}\ref{cond:2nd-4th-compare} holds for $L = \frac{\Gamma(5 / k)\Gamma(1/k)}{\Gamma(3/k)^2} + 1$. To see this, first notice \[
    \frac{\Gamma(5/k)^r \,\Gamma(3/k)^{2d - 2r}}{\Gamma(1/k)^{2d - r}} \leq \biggl[\frac{\Gamma(5 / k)\Gamma(1/k)}{\Gamma(3/k)^2}\biggr]^r \frac{\Gamma(3/k)^{2d}}{\Gamma(1/k)^{2d}},
    \] which by \cref{lem:moments-lp-sphere} further implies $\E(X_1^4\dotsm X_r^4 X_{r+1}^2 \dotsm X_{2d-r}^2) \leq (L-1)^r \E(X_1^2\dotsm X_{2d}^2)$. By $\E(X_1^2\dotsm X_{2d}^2) \to 1$ that we just verified, we obtain for large enough $n$ that \[
        \E(X_1^4\dotsm X_r^4 X_{r+1}^2 \dotsm X_{2d-r}^2) \leq L^r.
    \] % \dm{I'm a bit confused about condition (C) here. I think there are some more things to say. It would be better to rerrange this proof. Say what you prove and then prove it.}.

The proof of the case $k = \infty$ is similar to the above. Let $\mathbf Z \sim \Unif([-1,1]^n)$ and $X = R\frac{\mathbf Z}{\nm{\mathbf Z}_\infty}$. Note $\nm{\mathbf Z}_\infty$ is the maximum of $n$ i.i.d.\ $\Unif[0,1]$ random variables, and so $\E\nm{\mathbf Z}_\infty^{4d} = \frac{n}{4d+n}$. Going through the same computation, one can conclude that satisfying $(1/3)^d\E(R^{2d}) \to 1$ and $(1/3)^{2d}\Var(R^{2d}) \to 0$ would imply \cref{thm:main} \cref{cond:expec,cond:var}, and verify that \[\E(X_1^4\dotsm X_r^4 X_{r+1}^2\dotsm X_{2d-r}^2) \leq (9/5)^r \E(X_1^2\dotsm X_{2d}^2),\] which implies \cref{cond:2nd-4th-compare}. We omit the details.
\end{proof}

\printbibliography

\end{document}